\documentclass{amsart}

\usepackage[utf8]{inputenc} 
\usepackage[T1]{fontenc}
\usepackage{amsfonts}
\usepackage{amsmath}
\usepackage{amsthm}
\usepackage{amssymb}
\usepackage{latexsym}
\usepackage{enumerate}
\usepackage{tikz-cd}
\usepackage{caption}
 \usepackage{hyperref}
\usepackage{phoenician}
\usepackage{bbm}
\usepackage{enumitem}

\usetikzlibrary{fadings}
\usetikzlibrary{patterns}
\usetikzlibrary{shadows.blur}
\usetikzlibrary{shapes}

\newtheorem{theorem}{Theorem}[section]
\newtheorem{lemma}[theorem]{Lemma}
\newtheorem{proposition}[theorem]{Proposition}
\newtheorem{corollary}[theorem]{Corollary}

\newtheorem{remark}[theorem]{Remark}

\newtheorem{question}[theorem]{Question}

\newtheorem{claim}{Claim}

\newenvironment{proof2} {\begin{proof}[Proof of the claim]} {\end{proof}}
\newenvironment{proof3} {\begin{proof}[Proof of Proposition \ref{poprawka}]} {\end{proof}}

\theoremstyle{definition}
\newtheorem{definition}[theorem]{Definition}
\newtheorem{example}[theorem]{Example}

\newtheoremstyle{sltheorem}
{}                
{}                
{\slshape}        
{}                
{\bfseries}       
{.}               
{ }               
{\thmname{#1}\thmnote{ \bfseries #3}}                
\theoremstyle{sltheorem}

\theoremstyle{sltheorem}
\newtheorem{reptheorem}{Theorem}

\DeclareMathOperator{\at}{at}

\DeclareMathOperator{\bor}{Bor}
\DeclareMathOperator{\st}{St}
\DeclareMathOperator{\clop}{Clop}
\DeclareMathOperator{\cov}{cov}
\begin{document}

\newcommand{\cc}{\mathfrak{c}}
\newcommand{\N}{\mathbb{N}}
\newcommand{\B}{\mathbb{B}}
\newcommand{\A}{\mathbb{A}}
\newcommand{\C}{\mathbb{C}}
\newcommand{\E}{\mathbb{E}}
\newcommand{\Q}{\mathbb{Q}}
\newcommand{\R}{\mathbb{R}}
\newcommand{\K}{\mathbb{K}}
\newcommand{\mem}{\varrho}
\newcommand{\Z}{\mathbb{Z}}
\newcommand{\T}{\mathbb{T}}
\newcommand{\G}{\mathbb{G}}
\newcommand{\HH}{\mathbb{H}}
\newcommand{\F}{\mathbb{F}}
\newcommand{\PP}{\mathbb{P}}
\newcommand{\rin}{\right\rangle}
\newcommand{\SSS}{\mathbb{S}}
\newcommand{\forces}{\Vdash}
\newcommand{\dom}{\text{dom}}
\newcommand{\osc}{\text{osc}}
\newcommand{\FF}{\mathcal{F}}
\newcommand{\AAA}{\mathcal{A}}
\newcommand{\BB}{\mathcal{B}}
\newcommand{\I}{\mathcal{I}}
\newcommand{\X}{\mathcal{X}}
\newcommand{\Y}{\mathcal{Y}}
\newcommand{\MM}{\mathcal{M}}
\newcommand{\CC}{\mathcal{C}}
\newcommand{\OO}{\mathcal{\nu_\infty}}
\newcommand{\non}{\mathfrak{non}}
\newcommand{\add}{\mathfrak{add}}

\newcommand{\cof}{\mathfrak{cof}}
\newcommand{\downset}[1]{\langle #1 \rangle}

\author{Damian G\l odkowski}
\address[D.~G\l odkowski]{Institute of Mathematics of the Polish Academy of Sciences,
ul.  \'Sniadeckich 8,  00-656 Warszawa, Poland}
\address{Faculty of Mathematics, Informatics, and Mechanics, 
University of Warsaw, ul. Banacha 2, 02-097 Warszawa, Poland}
\email{\texttt{d.glodkowski@uw.edu.pl}}

\thanks{The research of the first named author was partially supported by the NCN (National Science Centre, Poland) research grant no. 2021/41/N/ST1/03682.}

\author{Agnieszka Widz}
\address[A.~Widz]{Institute of Mathematics, {\L}\'od\'z University of Technology, Aleje Politechniki 8, 93-590 {\L}\'od\'z, Poland}
\email{AgnieszkaWidzENFP@gmail.com}
\thanks{The research of the second named author was supported by the NCN (National Science Centre, Poland), under the Weave-UNISONO call in the Weave programme 2021/03/Y/ST1/00124.}
\thanks{For the purpose of Open Access, the authors have applied a CC-BY public copyright licence to any Author Accepted Manuscript (AAM) version arising from this submission.}


\title{Epic math battle of history: Grothendieck vs Nikodym}

\begin{abstract}
We define a $\sigma$-centered notion of forcing that forces the existence of a Boolean algebra with the Grothendieck property and without the Nikodym property. In particular, the existence of such an algebra is consistent with the negation of the continuum hypothesis. The algebra we construct consists of Borel subsets of the Cantor set and has cardinality $\omega_1$. We also show how to apply our method to streamline Talagrand's construction of such an algebra under the continuum hypothesis.
\end{abstract}
\maketitle
\section{Introduction}
In 1953, Grothendieck \cite[Section 4]{Grothendieck} proved that the space $l_\infty$ of bounded sequences has the following property:

\begin{center}
    \textit{All weak*-convergent sequences in the dual space $l_\infty^*$ are also weakly convergent.}
\end{center}
The above theorem motivated the following definition.

\begin{definition}
     A Banach space $X$ has the \textbf{Grothendieck property} if all weak*- convergent sequences in the dual space $X^*$ are also weakly convergent. 
\end{definition}

Research on the Grothendieck property has long history and is still ongoing \cite{Barcenas_Diomedes, Bourgain, Diestel_and_vector_measures,  Gonzalez_Onieva, Gonzalez_Kania, Romeo, Haydon_Boolean_rings, Kakol_Sobota_Zdomskyy, Koszmider_continuous_functions, Talagrand_Un_nouveau}. 
If $X$ is of the form $C(K)$ for a compact space $K$, then $X$ has the Grothendieck property if and only if each  weak*-convergent sequence of Radon measures on $K$ is also weakly convergent. Recall that $l_\infty$ is isometric to the Banach space $C(\beta\N)$ of continuous functions on the \v{C}ech-Stone compactification of the natural numbers. Moreover, $\beta \N$ is the Stone space of the Boolean algebra $\mathcal{P}(\N)$.

Schachermayer, inspired by the Grothendieck's result, 
introduced the notion of the Grothendieck property for Boolean algebras \cite[Definition 2.3]{Schachermayer}.

\begin{definition}
  A Boolean algebra $\A$ has the \textbf{Grothendieck property}, if the Banach space $C(\st(\A))$ of continuous functions on the Stone space of $\A$ has the Grothendieck property.
\end{definition}
 
Analogously, motivated by Nikodym's article \cite{Nikodym}, Schachermayer defined the Nikodym property for Boolean algebras \cite[Definition 2.4]{Schachermayer}.  

\begin{definition}
    A Boolean algebra $\A$ has the \textbf{Nikodym property}, if every sequence $(\mu_n)$ of bounded finitely additive signed measures on $\A$, which is pointwise convergent to zero (i.e. for all $A\in \A$ we have $\lim_{n \to \infty}\mu_n(A)=0$) is bounded in the norm (i.e. $\sup_{n\in\N} ||\mu_n||$ is bounded, cf. Section \ref{notation}).
\end{definition}
 The Nikodym property is similar to the Grothendieck property in many ways. For example, if a Boolean algebra $\A$ has the Grothendieck or Nikodym property, then its Stone space does not contain non-trivial convergent sequences. 
 In \cite{Grothendieck} and \cite{Ando} the authors proved that the complete Boolean algebras have both the Grothendieck and Nikodym properties. 
The completeness assumption can be relaxed to some combinatorial property (SCP) introduced by Haydon \cite[Definition 1A, Proposition 1B]{Haydon}, which is even weaker than  $\sigma$-completeness \cite{Freniche}. Other connections between the Grothendieck property and the Nikodym property may be found in \cite{Molto, Seever}. Both of the properties were also considered in a recent paper by \.Zuchowski in the context of filters on $\N$ \cite{zuchowski2024nikodym}.

 However, the Grothendieck and Nikodym properties are not equivalent. There are Boolean algebras with the Nikodym property but without the Grothendieck property, e.g. the Boolean algebra of Jordan measurable subsets of the unit interval \cite{Lopez_Alfonso},\cite[Propositions 3.2, 3.3]{Schachermayer}. The question if there is a Boolean algebra with the Grothendieck property, but without the Nikodym property turned out to be much more difficult. This is the central question of our article.

 \begin{question}
     Does there exist a Boolean algebra with the Grothendieck property that does not have the Nikodym property? 
 \end{question}

 So far there was only one known example of such a Boolean algebra. It was constructed by Talagrand in \cite{talagrand}. However, his construction uses the continuum hypothesis ({\sf CH}) and so the question of the existence of such a Boolean algebra in {\sf ZFC} remains open. Since Talagrand's construction there was no much progress in this subject, so it was natural to ask the following question. 

 \begin{question}
  Is it consistent with $\neg${\sf CH} that there is a Boolean algebra with the Grothendieck property but without the Nikodym property?   
 \end{question}

In this article we answer this problem in the affirmative. Moreover, the algebra we construct has cardinality $\omega_1$.

\begin{reptheorem}[\ref{main}]
It is consistent with $\neg${\sf CH} that there is a Boolean algebra of size $\omega_1$ with the Grothendieck property but without the Nikodym property. 
 \end{reptheorem}
 
Very recently, there has been released a preprint by Sobota and Zdomskyy \cite{SZpreprint} with a proof that Martin's axiom ({\sf MA}) implies the existence of such an algebra of cardinality $\mathfrak c$.\footnote{We were not aware of the existence of these results before the preprint appeared. Our results were obtained independently.}
 
The proof of Theorem \ref{main} strongly relies on the ideas behind Talagrand's construction. His Boolean algebra consists of Borel sets satisfying certain symmetry property (we call such sets \emph{balanced sets}). This ensures that the constructed Boolean algebra will not have the Nikodym property.
 
We define a $\sigma$-centered forcing notion that extends a given countable balanced Boolean algebra to a bigger one that is still balanced. Moreover, some sequences of measures (picked by a generic filter) which were weak*-convergent in the initial algebra are no longer weak*-convergent in the extension. Then we show that in the model obtained from the finite support iteration of length $\omega_1$ of such forcing notions, there exists a balanced Boolean algebra with the Grothendieck property. The idea behind this forcing comes from the work of Koszmider \cite{Koszmider_minimal_extensions} and of Fajardo \cite{fajardo}. In the former paper Koszmider introduced a notion of forcing that adds a Boolean algebra of cardinality $\omega_1$, whose Stone space does not contain non-trivial convergent sequences. Fajardo adapted this method to obtain a Banach space $C(K)$  of small density and with few operators. In particular, this space has the Grothendieck property. In our article we show how to combine this approach with the theory of balanced algebras to obtain a Boolean algebra without the Nikodym property.   

Most of the results concerning fundamental properties of balanced sets (cf. Section \ref{balanced-families-section}) that we use in our paper are essentially due to Talagrand. However, our construction requires some significant changes. Since the construction is rather complicated and technical, we decided to include detailed proofs at each step. We also show how to construct a balanced Boolean algebra with the Grothendieck property under {\sf CH} using our modification of Talagrand's (see Theorem \ref{main_CH}).

Another interesting related issue is the question about the possible sizes of Boolean algebras with the Grothendieck and Nikodym properties. There always exists such an algebra of size $\mathfrak{c}$ (e.g. $\mathcal{P}(\N)$). It is well-known that if $\mathfrak{p}=\mathfrak{c}$, then $\mathfrak{c}$ is the only possible size of such an algebra (it follows from \cite[Corollary 3F]{Haydon_Levy_Odell}). In particular, it happens under {\sf MA}. Brech showed the consistency of the existence of a Boolean algebra with the Grothendieck property of cardinality smaller than $\mathfrak{c}$ \cite{Brech}. In \cite[Chapter 52, Question 10]{Open_problems_II}
Koszmider asked whether it is consistent that there is no Boolean algebra with the Grothendieck property of size $\mathfrak p$. It turned out that the answer is positive \cite[Proposition 6.18]{Separably_injective}. Recently Sobota and Zdomskyy published several articles on cardinal characteristics related to Boolean algebras with the Grothendieck or Nikodym property \cite{Sobota_Applied_Logic, Sobota_kukuryku, SZ-Nikodym_in-Sacks, SZ_forcing, SZ_adding_real}. The Boolean algebra we construct is the first example of a Boolean algebra with the Grothendieck property and without the Nikodym property of size $\omega_1< \mathfrak c$ (in particular, our model satisfies $\mathfrak p = \omega_1 < \mathfrak c$, cf. Corollary \ref{pseudointersection_number}) and the first construction of such a Boolean algebra of size less than $\mathfrak{c}$. In particular, the Stone space of this algebra is another example of a Efimov space. In fact, if we want only to obtain a Boolean algebra with the Grothendieck property, then our forcing can be simplified in a natural way (by dropping some restrictions on the conditions).  

It is also worth mentioning that the Grothendieck and Nikodym properties are also discussed in the non-commutative setting in the category of C*-algebras. The definition of the Grothendieck property for C*-algebras is the same as for general Banach spaces. We say that a C*-algebra $A$ has the Nikodym property, if every sequence of bounded linear functionals on $A$ that is convergent to $0$ on projections is bounded in the norm. If $\A$ is a Boolean algebra, then it has the Nikodym property if and only if $C(\st(\A))$ has the Nikodym property, when considered as a C*-algebra. The Nikodym property is especially interesting in the case when given C*-algebra has many projections, e.g. when its real rank is zero. It is well-known that von Neumann algebras have both the Grothendieck and Nikodym properties (see \cite[Corollary 7]{pfitzner} and \cite[Theorem 1]{von_neumann_nikodym}). The problem of the existence of C*-algebras with the Grothendieck property and without the Nikodym property is still open in {\sf ZFC} even in the non-commutative case. 
\begin{question}
    Is there a C*-algebra of real rank zero, which has the Grothendieck property, but does not have the Nikodym property?
\end{question}
The structure of the paper is the following. The next section describes notation and terminology. In Section \ref{GN-section} we introduce the property $(\mathcal{G})$ of Boolean algebras and the property of being balanced. Then we show that they imply the Grothendieck property and the negation of the Nikodym property respectively. Section \ref{balanced-families-section} is devoted to properties of finite balanced families (this includes the behavior of balanced families under basic operations and approximating balanced families with families of clopen subsets of the Cantor set) and tools for extending countable balanced Boolean algebras. In Section \ref{construction} we show a method of extending countable balanced Boolean algebras to bigger ones in a way that destroys weak*-convergence of given sequences of measures. Then we show how to apply this method to construct a Boolean algebra with the Grothendieck property and without Nikodym property assuming the continuum hypothesis. In Section \ref{forcing} we describe a $\sigma$-centered notion of forcing that forces the existence of a Boolean algebra with the Grothendieck property and without the Nikodym property. In the last section we include final remarks and state some open questions.

\section{Notation}\label{notation}

\subsection*{Basic symbols.}
 For the purpose of this paper we will denote by $\N$ the set of positive integers. The smallest uncountable ordinal is denoted by $\omega_1$. The cardinality of the set of all real numbers is denoted by $\mathfrak c$. 
 
 For a sequence $s$, its $m$-th term will be denoted by $s_m$. By the Cantor set we mean the set $C=\{-1,1\}^\N$ of the sequences with values in $\{-1,1\}$ with the usual product topology. The set $\{-1,1\}^n$ consists of all sequences of length $n$ with values in $\{-1,1\}$. 

For $s\in \{-1,1\}^n $ we put $$\downset{s}=\{x\in C: x\upharpoonright n=s\},$$ where $x\upharpoonright n$ is the sequence of first $n$ elements of $x$. The family of all Borel subsets of $C$ will be denoted by $\bor(C)$. For a set $Z\subseteq C, \chi_Z$ stands for the characteristic function of $Z$.  

\subsection*{Boolean algebras.} 
For the basic terminology concerning Boolean algebras see \cite{DDLS, Introduction_to_Boolean_Algebras, koppelberg}. We will focus on Boolean algebras consisting of Borel subsets of the Cantor set endowed with the standard operations $\cup,\cap,\backslash$. The symmetric difference of sets $A$ and $B$ will be denoted by $A\triangle B$. For $n\in \N$, $\A_n$ is the finite subalgebra of $\bor(C)$ generated by $\{\downset{s}: s\in\{-1,1\}^n \}$. The family of all clopen subsets of $C$ will be denoted by $\clop(C)$. Note that $\clop(C)=\bigcup_{n\in\N} \A_n$. For a Boolean algebra $\A$ we denote by $\at(\A)$ the set of its atoms. In particular, $\at(\A_n)=\{\downset{s}:s\in \{-1,1\}^n\}$. The Stone space of $\A$ will be denoted by $\st(\A)$. For $A\in \A$ we denote by $[A]$ the corresponding clopen subset of $\st(\A)$.
A family $\{H_n\}_{n\in\N} \subseteq \A$ is called an antichain, if $H_n\cap H_m = \varnothing$ for $n\neq m$. For a Boolean algebra $\B$ and a subset $B\subseteq C$ we put
$$\FF(\B, B)= \{A\cap B, A\backslash B: A\in \B\}.$$

\subsection*{Measures.} We will consider measures on Boolean algebras and on their Stone spaces. For the general theory of such measures see \cite[Chapter V]{semadeni}. The symbol $\lambda$ will denote the normalized Haar measure on $C$ (considered as a group with coordinate-wise multiplication). In particular, $\lambda(\downset{s})=1/2^n$ for $s\in \{-1,1\}^n$. Throughout the article, we will discuss canonical measures (witnesses to the lack of the Nikodym property) $\varphi_n$ for $n\in \N$, given by the formula
$$\varphi_n (A) = \int_A \delta_n d\lambda$$
for $A\in \bor(C)$, where $\delta_n\colon C\rightarrow \{-1,1\}, \delta_n(x)=x_n$.

We will denote by $\mathcal{L}_2(C)$ the real Hilbert space of square-integrable (with respect to $\lambda$) functions on $C$ with the inner product 
$$\downset{f,g}=\int_C fg d\lambda.$$

In what follows a measure on a Boolean algebra $\A$ is always a finitely additive signed bounded measure on $\A$. We will denote them by the lowercase Greek letters $\mu,\nu, \vartheta$, and we will call them concisely ``measures on $\A$''. If $\mu$ is a measure on a Boolean algebra $\A$ and $\B\subseteq \A$ is a subalgebra, then $\mu\upharpoonright \B$ denotes the restriction of $\mu$ to $\B$. For a measure $\mu$ on $\A$ we define its variation $|\mu|$ as a measure on $\A$ given by 
$$|\mu|(X) = \sup\{|\mu(A)|+|\mu(B)|: A,B\in\A, A,B\subseteq X, A\cap B = \varnothing\},$$
and its norm (total variation) as 
$$\|\mu\| = |\mu|(1),$$
where $1$ is the biggest element of $\A$. Note that for every $n\in\N$ we have $|\varphi_n|=\lambda$ and $\|\varphi_n\|=1$. If $\mu$ is non-negative and $\|\mu\|=1$, then $\mu$ is called a probability measure.

For a compact space $K$ we denote by $M(K)$ the Banach space of Radon measures on $K$, endowed with the total variation norm. Every measure on a Boolean algebra $\A$ extends uniquely to a Radon measure on the space $\st(\A)$ (cf. \cite[Section 18.7]{semadeni}). If $\mu$ is a measure on a Boolean algebra $\A$, then $\widetilde{\mu}$ denotes the corresponding Radon measure on $\st(\A)$. In particular, $\|\widetilde{\mu}\|=\|\mu\|$ and $|\widetilde{\mu}|=|\mu|$. For a Radon measure $\widetilde{\mu}$ a Borel set $F\subseteq \st(\A)$ is a Borel support of $\widetilde{\mu}$, if $\widetilde{\mu}(X)=0$ for every Borel $X\subseteq \st(\A)\backslash F$. We say that a sequence $(\widetilde{\mu}_n)_{n\in\N}$ of Radon measures has pairwise disjoint Borel supports, if there are pairwise disjoint Borel sets $(F_n)_{n\in\N}\subseteq \st(\A)$ such that $F_n$ is a Borel support of $\widetilde
{\mu}_n$ for every $n\in \N$. Note that, unlike the support of measure, a Borel support is not unique.

We will consider 3 types of convergence of sequences of measures on a compact space $K$. We say that a sequence $(\widetilde{\mu}_n)_{n\in \N}\subseteq M(K)$ converges weakly, if it is convergent in the weak topology of the Banach space $M(K)$. We say that $(\widetilde{\mu}_n)_{n\in \N}\subseteq M(K)$ is weak*-convergent, if it converges in the weak* topology, where $M(K)$ is treated as the dual space to the Banach space of continuous functions $C(K)$. A sequence $(\mu_n)_{n\in\N}$ of measures on a Boolean algebra $\A$ is said to be pointwise convergent, if there is a measure $\mu$ on $\A$ such that $\mu_n(A)\xrightarrow{n\to \infty} \mu(A)$ for every $A\in \A$. It is a well-known fact, that a sequence of Radon measures $(\widetilde{\mu}_n)_{n\in \N}$ on $\st(\A)$ is weak*-convergent if and only if the sequence $(\mu_n)_{n\in\N}$ is bounded in the norm and pointwise convergent on $\A$. 

\subsection*{Forcing.} Most of the notation concerning forcing should be standard. For the unexplained terminology see \cite{Bartoszynski, jech, kunen}. The universe of sets will be denoted by $V$. For a forcing notion $\PP$ we denote by $V^\PP$ a generic extension of $V$ obtained by forcing with $\PP$. The evaluation of a constant $c$ in the class $V$ is denoted by $c^V$. For the purpose of the section involving forcing we will identify Borel subsets of $C$ with their codes with respect to some absolute coding (cf. \cite[Section 25]{jech}).

\section{Grothendieck and Nikodym properties}\label{GN-section}

 In this section we will reduce the problem of the existence of a Boolean algebra with the Grothendieck property and without the Nikodym property by introducing the property $(\mathcal{G})$ and the notion of a balanced Boolean algebra.  

We start with the notion of semibalanced sets that describes these subsets $A\subseteq C$ for which the occurrences of $1$'s and $-1$'s at $r$-th coordinate of elements of $A$ are almost equally distributed for large enough $r$. 

\begin{definition}\label{semi_balanced_def}
 Let $A\in \bor(C); m\in \N; \varepsilon>0$. We say that $A$ is $(m,\varepsilon)$\textbf{-semibalanced} if   
\begin{gather}\label{3.0.1}
\tag{3.1.1}
\forall{r>m} \ |\varphi_r(A)|< \frac{\varepsilon}{r}.
\end{gather}
 We say that $A$ is \textbf{semibalanced}, if for every $\varepsilon>0$ there is $m\in\N$ such that $A$ is $(m,\varepsilon)$-semibalanced.
\end{definition}

\begin{definition}\label{balanced_def}
    Let $A\in \bor(C); m,t\in\N;t\geqslant m;\varepsilon>0$. We say that $A$ is $(m,t,\varepsilon)$\textbf{-balanced} if for every $s\in \{-1,1\}^m$
   
\begin{gather}\label{3.1.1}
\tag{3.2.1}
\frac{\lambda(A\cap \downset{s})}{\lambda(\downset{s})}< \frac{\varepsilon}{m} \ \text{or} \ \frac{\lambda(\downset{s}\backslash A)}{\lambda(\downset{s})}< \frac{\varepsilon}{m}
\end{gather}
and
\begin{gather}\label{3.1.2}
\tag{3.2.2}
\forall{r\in (m, t]} \
\frac{|\varphi_r(A\cap \downset{s})|}{\lambda(\downset{s})}< \frac{\varepsilon}{r}.
\end{gather}
We say that A is $(m,\varepsilon)$\textbf{-balanced} if for every $s\in \{-1,1\}^m$ the condition (\ref{3.1.1}) is satisfied and
\begin{gather}\label{3.1.3}
\tag{3.2.3}
\forall{r>m} \
\frac{|\varphi_r(A\cap \downset{s})|}{\lambda(\downset{s})}< \frac{\varepsilon}{r}.
\end{gather}

\end{definition}

\begin{definition}
We say that a finite subfamily $\AAA\subseteq \bor(C)$ is $(m,\varepsilon)$\textbf{-balanced} if every $A\in \AAA$ is $(m,\varepsilon)$-balanced. 

We say that a family $\B\subseteq \bor(C)$ is \textbf{balanced}, if for every finite subfamily $\AAA\subseteq \B$ and every $\varepsilon> 0$ there is $m\in \N$ such that
$\AAA$ is  $(m,\varepsilon)$-balanced. A set $A\in \bor(C)$ is $\textbf{balanced}$, if $\{A\}$ is balanced. 
\end{definition}

\begin{remark}\label{dla_dobra_nauki}
    If $\AAA$ is a finite balanced family and $\varepsilon>0$, then for every $n\in \N$ there is $m>n$ such that $\AAA$ is $(m,\varepsilon)$-balanced.
\end{remark}

Note that if $\B$ is balanced, then every member of $\B$ is balanced, but the reverse implication does not hold in general. It may happen that $A$ and $B$ are balanced, while the sets $\{m\in \N: A $ is $(m,\varepsilon)$-balanced$\}$ and $\{m\in\N: B $ is $(m,\varepsilon)$-balanced$\}$ are disjoint for some $\varepsilon>0$.

\begin{lemma}\label{balanced+semibalanced}
    Let $A\in \bor(C),\varepsilon>0 $ and $m,t\in \N$, where $t\geqslant m$. Then $A$ is $(m,\varepsilon)$-balanced if and only if $A$ is $(m,t,\varepsilon)$-balanced and for every $s\in \{-1,1\}^m$ the set $A\cap\downset{s}$ is $(t,2^{-m}\varepsilon)$-semibalanced. In particular, if $A$ is balanced, then it is semibalanced. 
\end{lemma}
\begin{proof}
If $A$ is $(m,\varepsilon)$-balanced, then it is clearly $(m,t,\varepsilon)$-balanced and by (\ref{3.1.3}) for any $s\in\{-1,1\}^m$ and $r>t$ we have
$$|\varphi_r(A\cap \downset{s})| < \lambda(\downset{s})\frac{\varepsilon}{r}=\frac{\varepsilon}{2^mr},$$
which shows that $A\cap \downset{s}$ is $(t,2^{-m}\varepsilon)$-semibalanced.

The above inequality also shows that if $A$ is $(t,2^{-m}\varepsilon)$-semibalanced, then (\ref{3.1.3}) is satisfied for $s\in\{-1,1\}^m$ and $r>t$. In particular, if  $A$ is $(m,t,\varepsilon)$-balanced and $(t,2^{-m}\varepsilon)$-semibalanced, then it is $(m,\varepsilon)$-balanced.

To see that any balanced set is semibalanced fix $\varepsilon>0$ and $m\in\N$ such that $A$ is $(m,\varepsilon)$-balanced. From the first part of the lemma applied to $t=m$ we get that for $r>m$ 
$$|\varphi_r(A)|\leqslant \sum_{s\in\{-1,1\}^m} |\varphi_r(A\cap\downset{s})| < \sum_{s\in\{-1,1\}^m} \frac{\varepsilon}{2^mr}=\frac{\varepsilon}{r},$$
so $A$ is semibalanced. 
\end{proof}

We will present a few examples to illustrate the above definitions.
\begin{example}
    Every clopen subset of $C$ is $(m,\varepsilon)$-balanced for every $\varepsilon>0$ and sufficiently large $m\in \N$.
\end{example}

\begin{example}
We will construct an open balanced set $U$ that is not clopen. More precisely, $U$ will be $\left(2^n,\frac{2^{n+2}}{2^{2^n}}\right)$-balanced for every $n \in \N$. 

For every $n\in \N$ consider the set $Z_n$ of all sequences of the form $(a_1, a_2, \dots, a_{2^n})$ of length $2^n$ with values in $\{-1, 1\}$ with the following properties:
\begin{enumerate}
    \item if $n=1$, then $a_1=-a_2$,
    \item if $n\neq 1$, then $a_1=a_2$,
    \item $\forall{l < n} \ a_{2^{l-1}+1}=a_{2^{l-1}+2}=\dots = a_{2^l}$,
    \item $a_{2^{n-1}+1}=a_{2^{n-1}+2}=\dots = a_{2^n-1}=-a_{2^{n}}$,
\end{enumerate}
and put $Z=\bigcup_{n\in\N} Z_n$.
On the figure below, the red (dark) sets are of the form $\downset{s}$, where $s\in Z_n$ and $s_{2^n}=1$ for some $n\in \N$, while blue (light) sets are of the form $\downset{s}$, where $s\in Z_n, s_{2^n}=-1$.

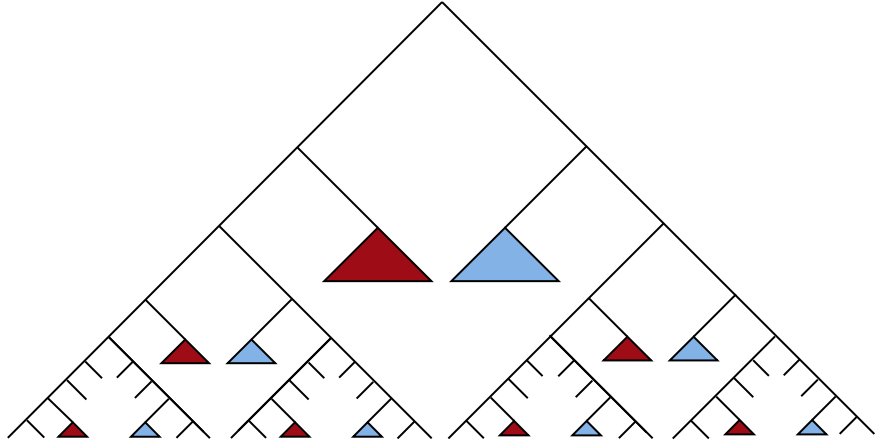
\begin{center}
\tikzset{every picture/.style={line width=0.75pt}} 

\begin{tikzpicture}[x=0.75pt,y=0.75pt,yscale=-1,xscale=1]

\draw    (292.19,523.55) -- (510.2,741.56) ;
\draw    (292.19,523.55) -- (73.18,743.52) ;
\draw    (219.17,596.88) -- (259.78,637.48) ;
\draw    (364.94,596.47) -- (323.94,637.47) ;
\draw    (179.74,636.56) -- (286.95,743.77) ;
\draw    (142.62,673.82) -- (162.67,693.87) ;
\draw    (216.33,673.59) -- (196.05,693.87) ;
\draw  [fill={rgb, 255:red, 158; green, 12; blue, 22 }  ,fill opacity=1 ] (259.78,637.48) -- (286.83,664.43) -- (232.72,664.43) -- cycle ;
\draw  [fill={rgb, 255:red, 130; green, 178; blue, 230 }  ,fill opacity=0.62 ] (323.94,637.47) -- (351,664.42) -- (296.88,664.42) -- cycle ;
\draw  [fill={rgb, 255:red, 130; green, 178; blue, 230 }  ,fill opacity=1 ] (196.05,693.87) -- (208.06,705.84) -- (184.04,705.84) -- cycle ;
\draw  [fill={rgb, 255:red, 158; green, 12; blue, 22 }  ,fill opacity=1 ] (162.67,693.87) -- (174.68,705.84) -- (150.65,705.84) -- cycle ;
\draw    (404.17,635.03) -- (295.37,743.83) ;
\draw    (366.33,673.14) -- (385.67,692.47) ;
\draw    (440.43,671.09) -- (419.05,692.47) ;
\draw  [fill={rgb, 255:red, 130; green, 178; blue, 230 }  ,fill opacity=1 ] (419.05,692.47) -- (431.06,704.44) -- (407.04,704.44) -- cycle ;
\draw  [fill={rgb, 255:red, 158; green, 12; blue, 22 }  ,fill opacity=1 ] (385.67,692.47) -- (397.68,704.44) -- (373.65,704.44) -- cycle ;
\draw    (124,692.56) -- (172,740.56) ;
\draw    (346.86,692.34) -- (395.27,740.75) ;
\draw    (236.22,693.04) -- (189.36,739.9) ;
\draw    (460.09,692.31) -- (408.54,743.86) ;
\draw    (439.33,713.2) -- (449.13,723) ;
\draw    (481.71,714.11) -- (473.32,722.5) ;
\draw    (449.08,704.32) -- (457,712.24) ;
\draw    (490.71,722.61) -- (478.82,734.5) ;
\draw    (472.53,703.92) -- (464.5,711.95) ;
\draw  [fill={rgb, 255:red, 130; green, 178; blue, 230 }  ,fill opacity=1 ] (478.82,734.5) -- (486.03,741.77) -- (471.61,741.77) -- cycle ;
\draw    (501.33,732.92) -- (492.73,741.52) ;
\draw    (430.47,722.7) -- (442.26,734.49) ;
\draw  [fill={rgb, 255:red, 158; green, 12; blue, 22 }  ,fill opacity=1 ] (442.26,734.49) -- (449.47,741.76) -- (435.05,741.76) -- cycle ;
\draw    (346.37,691.72) -- (398.38,743.74) ;
\draw    (325.61,713.7) -- (335.41,723.5) ;
\draw    (367.99,714.61) -- (359.6,723) ;
\draw    (334.96,704.42) -- (342.88,712.34) ;
\draw    (376.99,723.11) -- (365.1,735) ;
\draw    (358.81,704.42) -- (350.78,712.45) ;
\draw  [fill={rgb, 255:red, 130; green, 178; blue, 230 }  ,fill opacity=1 ] (365.1,735) -- (372.31,742.27) -- (357.89,742.27) -- cycle ;
\draw    (389.81,735.22) -- (381.36,743.67) ;
\draw    (316.75,723.2) -- (328.54,734.99) ;
\draw  [fill={rgb, 255:red, 158; green, 12; blue, 22 }  ,fill opacity=1 ] (328.54,734.99) -- (335.75,742.26) -- (321.33,742.26) -- cycle ;
\draw    (235.87,693.47) -- (185.72,743.62) ;
\draw    (215.11,714.37) -- (224.91,724.17) ;
\draw    (257.49,715.27) -- (249.1,723.67) ;
\draw    (224.86,705.48) -- (232.78,713.4) ;
\draw    (266.49,723.77) -- (254.6,735.67) ;
\draw    (248.31,705.09) -- (240.28,713.12) ;
\draw  [fill={rgb, 255:red, 130; green, 178; blue, 230 }  ,fill opacity=1 ] (254.6,735.67) -- (261.81,742.94) -- (247.39,742.94) -- cycle ;
\draw    (278.31,735.23) -- (269.72,743.82) ;
\draw    (206.25,723.87) -- (218.04,735.65) ;
\draw  [fill={rgb, 255:red, 158; green, 12; blue, 22 }  ,fill opacity=1 ] (218.04,735.65) -- (225.25,742.93) -- (210.83,742.93) -- cycle ;
\draw    (124,692.56) -- (175.11,743.67) ;
\draw    (103.01,714.38) -- (112.81,724.18) ;
\draw    (145.68,715) -- (137.28,723.39) ;
\draw    (112.24,704.98) -- (120.68,713.42) ;
\draw    (154.39,723.78) -- (142.5,735.68) ;
\draw    (136.21,705.1) -- (128.18,713.13) ;
\draw  [fill={rgb, 255:red, 130; green, 178; blue, 230 }  ,fill opacity=1 ] (142.5,735.68) -- (149.71,742.95) -- (135.29,742.95) -- cycle ;
\draw    (166.78,734.96) -- (158.22,743.52) ;
\draw    (93.38,723.51) -- (105.94,735.67) ;
\draw  [fill={rgb, 255:red, 158; green, 12; blue, 22 }  ,fill opacity=1 ] (105.94,735.67) -- (113.15,742.94) -- (98.73,742.94) -- cycle ;
\draw    (194.82,734.94) -- (203.41,743.53) ;
\draw    (82.82,734.64) -- (91.55,743.37) ;
\draw    (304.16,734.62) -- (313.26,743.72) ;
\draw    (417.61,734.7) -- (426.73,743.83) ;
\end{tikzpicture}
\end{center}

Let  $U= \bigcup_{s\in Z} \downset{s}$. Consider three cases. \\
1. If $s\in Z$ then $$m \frac{\lambda(\downset{s}\backslash U)}{\lambda(\downset{s})}=m \frac{\lambda(\varnothing)}{\lambda(\downset{s})}=0.$$
2. If $s \in \{-1, 1\}^{2^n}$ for $n\in \N \setminus \{1\}$ satisfy conditions (2), (3) and 
\begin{enumerate}
    \item[(4')]$ \ a_{2^{n-1}+1}=a_{2^{n-1}+2}=\dots = a_{2^n-1}=a_{2^{n}}$,
\end{enumerate}
 then there are only 4 nonempty sets of the form $U\cap\downset{s^{i}}$, for $i \in \{1, 2, 3, 4\}$ where $s^i\in \{-1,1\}^{2^{n+1}}$ and $\downset{s^{i}}\subseteq \downset{s}$. Hence

$$
m  \frac{\lambda(U\cap \downset{s})}{\lambda(\downset{s})}< m \frac{\sum_{i=1}^4\lambda(\downset{s^i})}{\lambda(\downset{s})}
\leqslant 2^n \frac{4\cdot 2^{-2^{n+1}}}{2^{-2^n}}=\frac{2^{n+2}}{2^{2^n}}.
$$
3. In other cases $$m \frac{\lambda(U\cap \downset{s})}{\lambda(\downset{s})}=m \frac{\lambda(\varnothing)}{\lambda(\downset{s})}=0.$$

For every $r\in \N$ the distribution of $1$'s and $-1$'s in the elements of $U$ at the $r$-th coordinate is symmetric, as can be easily seen in the figure - the blue (light) sets are symmetric to the red (dark) ones (i.e. $\downset{s}$ is a blue set if and only if $\downset{-s}$ is red). Thus, for every $r\in \N$ we have $$\frac{|\varphi_r(U\cap \downset{s})|}{\lambda(\downset{s})}=0.$$
\end{example}

\begin{proposition}\label{NN_implies_not_Nikodym}
    If a Boolean algebra $\B\subseteq \bor(C)$ consists of semibalanced sets, then $\B$ does not have the Nikodym property. In particular, if $\B$ is balanced, then it does not have the Nikodym property.
\end{proposition}

\begin{proof}
    Consider a sequence of measures $(\mu_n)_{n\in\N}$ on $\B$ given by
    $$\mu_n(A)=n\varphi_n(A).$$

    This sequence is pointwise convergent to 0, i.e.  $\lim_{n\to \infty}|\mu_n(A)|=0$ for all $A\in \B$. Indeed, for $A\in \B$ and for any $\varepsilon>0$ there exist $m\in \N$ such that for all $n> m$ by (\ref{3.0.1}) we have
 \begin{eqnarray*}
     |\mu_n(A)|=n |\varphi_n(A)|\leqslant n\frac{\varepsilon}{n}=\varepsilon.
\end{eqnarray*}
 However, $(\mu_n)_{n\in\N}$ is not bounded in the norm, because $$\sup_{n\in\N} ||\mu_n|| = \sup_{n\in\N} n\|\lambda\|=\infty.$$ In particular, $\B$ does not have the Nikodym property.

 If $\B$ is balanced, then by Lemma \ref{balanced+semibalanced} it consists of semibalanced sets, so it does not have the Nikodym property by the first part of the lemma.
\end{proof}

\begin{example}
The following Boolean algebra considered by Plebanek\footnote{Personal communication} is interesting in the context of our considerations:
$$ \B_P = \{B \in \bor(C): \lim_{n \to \infty} n\psi_n(B)=0 \},$$
     where $\psi_n(B)=\min\{\lambda(B\triangle A): A\in \A_n\}$. 

Each element of $\B_P$ is semibalanced, so by Proposition \ref{NN_implies_not_Nikodym} $\B_P$ does not have the Nikodym property. Indeed, one needs first to observe that if $B\in \B_P$ and $n<m$, then 
         $$\varphi_m(B) \leqslant \psi_n(B).$$
To show that every $B\in \B_P$ is semibalanced take any $B \in \B_P$ and $\varepsilon>0$. Since $ n \psi_n(B)$ converges to $0$, we can find $m\in \N$ such that for every $r\geqslant m$ we have $r\psi_r(B)<\frac{\varepsilon}{2}$. Then for every $r>m$ 
         $$
         r|\varphi_r(B)|\leqslant r\psi_{r-1}(B)\leqslant 2(r-1)\psi_{r-1}(B)< \varepsilon.
         $$
However, $\B_P$ is not balanced. To see this, take
    $$B= \bigcup_{k=2}^\infty \downset{s_k},$$
    where $s_k \in \{-1, 1\}^k$ is of the form $s_k=(-1,\dots,-1, 1, 1)$ for $k\geqslant 2$. Then 
    $$n\psi_n(B)=n\sum_{k>n}\lambda(\downset{s_k})=\frac{n}{2^n}$$
    converges to $0$, so $B\in \B_P$. But $B$ is not balanced. Indeed, for $k\in \N$ consider $s'_k \in \{-1, 1\}^k$ of the form $s'_k=(-1,\dots,-1, 1)$. Then 
    $$\frac{\lambda(B\cap \downset{s'_k})}{\lambda(\downset{s'_k})} =\frac{\lambda( \downset{s'_k}\setminus B)}{\lambda(\downset{s'_k})} = \frac{1}{2}.$$

There are no non-trivial convergent sequences in $\st(\B_P)$. However, this Boolean algebra does not have the Grothendieck property.  The sequence of measures on $\B_P$ given by
$$\vartheta_n(A) = 2^n \varphi_{2^n}(A \cap \downset{s'_n}),$$
(where $s'_n $ is as above) is weak*-convergent, but it is not weakly convergent\footnote{The idea behind this sequence is due to Avil\'es.}. 

Moreover, according to Borodulin-Nadzieja\footnote{Personal communication.} no semibalanced Boolean algebra containing $\B_P$ has the Grothendieck property.

\end{example}

In order to take care of the Grothendieck property it is enough to restrict the choice of sequences of measures to those with pairwise disjoint Borel supports and norms equal to 1. 

\begin{definition}\label{normal_sequence}
    Let $\A$ be a Boolean algebra. We say that a sequence $(\nu_n)_{n\in\N}$ of measures on $\A$ is \textbf{normal}, if:
     \begin{itemize}
        \item $\forall n\in\N \ \|\nu_n\|=1$,
        \item $(\widetilde{\nu}_n)_{n\in\N}$ has pairwise disjoint Borel supports.
    \end{itemize}
\end{definition}

The following definition will be important throughout the article. 

\begin{definition}\label{cisza_przed_burza}
    We say that a Boolean algebra $\B$ satisfies $(\mathcal{G})$
    if for every normal sequence $(\nu_n)_{n\in\N}$ of measures on $\B$ 
there are $G\in\B$, an antichain $\{H_0^n, H_1^n : n\in\N\}\subseteq \B$ and strictly increasing sequences $ (a_n)_{n\in\N}, (b_n)_{n\in\N} $ of natural numbers such that for all $n\in \N$
\begin{enumerate}[label=(\alph*)]
     \item $ G\cap H^{n}_1=\varnothing $,
    \item $ |\nu_{a_n}|(H^{n}_0)\geqslant 0.9$ and $|\nu_{b_n}|(H^{n}_1)\geqslant 0.9$,
    \item $ |\nu_{a_n}(G\cap H^{n}_0)|\geqslant 0.3 $.  
\end{enumerate}
\end{definition}

Note that Schachermayer introduced the property (G) in \cite{Schachermayer} as a name for the Grothendieck property. Our property ($\mathcal{G}$) is different. It implies the Grothendieck property, but the reverse implication does not hold.

To show that the property $(\mathcal{G})$ implies the Grothendieck property we will use the following lemma known as The Kadec-Pe\l czy\'nski-Rosenthal Subsequence Splitting Lemma:

\begin{lemma}\cite[Lemma 5.2.7]{albiac-kalton}\label{KPR}
Let $K$ be a compact space. For every bounded sequence $(\widetilde{\nu}_n)_{n\in\N}\subseteq M(K)$ there exists a non-negative real $r$ and a subsequence $(\widetilde{\nu}_{n_k})_{k\in\N}$, each
element of which may be decomposed into a sum of two measures $\widetilde{\nu}_{n_k}=\widetilde{\mu}_k+\widetilde{\theta}_k$, where $\widetilde{\mu}_k,\widetilde{\theta}_k\in M(K)$, 
 satisfying the following conditions:
 \begin{enumerate}
     \item the measures $\widetilde{\mu}_k$ are supported by pairwise disjoint Borel sets,
    \item $(\widetilde{\theta}_k)_{k\in\N}$ is weakly convergent,
    \item $||\widetilde{\mu}_k||=r$, for every $k \in \N.$
 \end{enumerate}
\end{lemma}

\begin{proposition}\label{G_implies_Grothendieck}
    If a Boolean algebra $\B$ satisfies $(\mathcal{G})$, then $\B$ has the Grothendieck property. 
\end{proposition}

\begin{proof}
    Suppose $\B$ satisfies the property $(\mathcal{G})$, but does not have the Grothendieck property. That is, there is a sequence $\widetilde{\mu}_n$ of measures on $\st(\B)$  weak*-convergent to a measure $\widetilde{\mu}$, which is not weakly convergent. Without loss of generality, by passing to a subsequence, we can assume that no subsequence of $(\widetilde{\mu}_n)_{n\in\N}$ is weakly convergent. Indeed, suppose that each subsequence of $(\widetilde{\mu}_n)_{n\in\N}$ contains a weakly convergent subsequence. Since the sequence is weak*-convergent to $\widetilde{\mu}$, such a subsequence must be also weakly convergent to $\widetilde{\mu}$. Then each subsequence of $(\widetilde{\mu}_n)_{n\in\N}$ has a subsequence weakly convergent to $\widetilde{\mu}$, so the whole sequence is weakly convergent to $\widetilde{\mu}$, which gives a contradiction. 
    
    Since $(\widetilde{\mu}_n)_{n\in\N}$ is weak*-convergent, it is bounded in the norm (cf. \cite[Theorem 3.88]{banach-space-theory}).   
By Lemma \ref{KPR} we can find a real $r$ and a subsequence $(\widetilde{\mu}_{n_k})_{k\in\N}$ each
element of which may be decomposed into the sum of two measures $\widetilde{\mu}_{n_k}=\widetilde{\nu}_k+\widetilde{\theta}_k$
 satisfying
 \begin{enumerate}
     \item the measures $\widetilde{\nu}_k$ are supported by pairwise disjoint Borel sets,
    \item $(\widetilde{\theta}_k)_{k\in\N}$ is weakly convergent,
    \item $||\widetilde{\nu}_k||=r$, for every $k \in \N.$ 
 \end{enumerate}

Note that since $(\widetilde{\theta}_n)_{n\in\N}$ is weakly convergent, $(\widetilde{\nu}_n)_{n\in\N}$ is not. In particular, $r\neq 0$. 
Thus, without loss of generality, by the normalization, we can assume that $r=1$. Then the sequence $(\nu_n)_{n\in\N}$ is normal. The Boolean algebra $\B$ satisfies the property $(\mathcal{G})$, so there are $G\in\B$, an antichain $\{H_0^n, H_1^n : n\in\N\}\subseteq \B$ and strictly increasing sequences $ (a_n)_{n\in\N}, (b_n)_{n\in\N} $ of natural numbers such that  
for all $n\in\N $ 
\begin{enumerate}[label=(\alph*)]
     \item $ G\cap H^{n}_1=\varnothing $,
    \item $ |\nu_{a_n}|(H^{n}_0)\geqslant 0.9$ and $|\nu_{b_n}|(H^{n}_1)\geqslant 0.9$,
    \item $ |\nu_{a_n}(G\cap H^{n}_0)|\geqslant 0.3 $.  
\end{enumerate}
 Hence for $n\in\N$ we have
\begin{enumerate}
    \item $|\nu_{a_n}(G)|\geqslant |\nu_{a_n}(G\cap H_0^n)|-|\nu_{a_n}|(G\backslash H_0^n)\geqslant 0.3-0.1=0.2$,
    \item $ |\nu_{b_n}|(G)\leqslant |\nu_{b_n}|(C\backslash H_1^n)\leqslant 0.1 $. 
\end{enumerate}
So there is no $\nu$ such that $\nu_{n}(G) \to \nu(G)$. Thus, $(\widetilde{\nu}_{k})_{k\in\N}$ is not weak*-convergent, which is a contradiction, since $(\widetilde{\nu}_{k})_{k\in\N}=(\widetilde{\mu}_{n_k}-\widetilde{\theta}_{k})_{k\in\N}$ is a difference of two weak*-convergent sequences. 
\end{proof}

Now we will introduce the property $(\mathcal{G}^*)$ similar to the property $(\mathcal{G})$, which focuses on only one sequence of measures. Then we will show that having the property  $(\mathcal{G}^*)$ for enough many sequences of measures we can conclude that the property $(\mathcal{G})$ holds.

\begin{definition}\label{G*-definition}
    Let $\B^*\subseteq \B$ be Boolean algebras and let $\nu=(\nu_n)_{n\in\N}$ be a sequence of measures on $\B^*$. We say that $(\B^*,\B, \nu)$ satisfies $(\mathcal{G}^*)$, if there are: an antichain $\{H_0^n, H_1^n : n\in\N\}\subseteq \B^*$, a set $G\in\B $ and strictly increasing sequences $ (a_n)_{n\in\N}, (b_n)_{n\in\N} $ of natural numbers such that for all $n\in\N$
    \begin{enumerate}[label=(\alph*)]
    \item $G\cap H_0^n \in \B^*$,
     \item $G\cap H_1^n = \varnothing $,
    \item $|\nu_{a_n}|(H_0^n)\geqslant 0.9$ and $|\nu_{b_n}|(H_1^n) \geqslant 0.9$,
    \item $|\nu_{a_n}(G\cap H_0^n)|\geqslant 0.3 $.
\end{enumerate}
\end{definition}

The next proposition shows the relationship between the properties $(\mathcal{G})$ and $(\mathcal{G}^*)$.

\begin{proposition}\label{thm-for-forcing-version}
    Suppose that $\B$ is a Boolean algebra such that for every normal sequence of measures $\nu=(\nu_n)_{n\in\N}$ on $\B$ 
    there is a subalgebra $\B^*\subseteq \B$ such that the sequence $(\nu_n\upharpoonright{\B^*})_{n\in \N}$ is normal and $(\B^*,\B,\nu\upharpoonright{\B^*}$) satisfies $(\mathcal{G}^*)$. Then $\B$ satisfies $(\mathcal{G})$. In particular, $\B$ has the Grothendieck property.
\end{proposition}

\begin{proof}
    Fix any normal sequence $(\nu_n)_{n\in\N}$ of measures on $\B$. Pick $\B^* \subseteq \B$ such that  there exist an antichain $\{H_0^n, H_1^n : n\in\N\}\subseteq \B^*$, $G \in \B$ and sequences $ (a_n)_{n\in\N}, (b_n)_{n\in\N} $ such that for the sequence $(\nu_{n}\upharpoonright{\B^*})$ the conditions (a)-(d) from Definition \ref{G*-definition} are satisfied. In particular, $G\cap H_1^n = \varnothing $, i.e. (a) of Definition \ref{cisza_przed_burza} holds.
    To see that (b) of Definition \ref{cisza_przed_burza} is satisfied, observe that
 \begin{eqnarray*}
         |\nu_{a_n}|(H^n_0)&=& \sup\{|\nu_{a_n}(A)|+|\nu_{a_n}(B)|: A, B \in \B, A, B \subseteq H^n_0, A\cap B = \varnothing\}\geqslant \\
         &\geqslant&\sup\{|\nu_{a_n}(A)|+|\nu_{a_n}(B)|: A, B \in \B^*, A, B \subseteq H^n_0, A\cap B=\varnothing\}=\\ &=&|\nu_{a_n}\upharpoonright{\B^*}| (H^n_0)\geqslant 0.9
    \end{eqnarray*}
   and similarly $|\nu_{b_n}|(H^n_1) \geqslant |\nu_{b_n}\upharpoonright{\B^*}| (H^n_1) \geqslant 0.9$.
   
   For Definition \ref{cisza_przed_burza} (c) note that $$|\nu_{a_n}(G\cap H_0^n)| = |\nu_{a_n}\upharpoonright{\B^*}(G\cap H_0^n)|\geqslant 0.3.$$ Use Proposition \ref{G_implies_Grothendieck} to conclude that $\B$ has the Grothendieck property.
\end{proof}

For the reader's convenience we provide a brief sketch of our constructions. We describe consecutive steps of reasoning, starting from general motivations. The parts devoted only to the case of construction under the continuum hypothesis are tagged ({\sf CH}), while the parts devoted to the forcing construction are tagged (F). 
\subsection*{Construction roadmap} 
\begin{enumerate}[leftmargin=18pt]
    \item General idea: we construct an increasing sequence $(\B_\alpha)_{\alpha<\omega_1}$ of balanced countable subalgebras of $\bor(C)$.
\begin{itemize}
    \item [({\sf CH})] For every $\alpha<\omega_1$ the triple $(\B^*_\alpha, \B_{\alpha+1}, (\nu_n^\alpha)_{n\in\N})$ satisfies $(\mathcal{G}^*)$, where $(\nu_n^\alpha)_{n\in\N}$ is a sequence of measures (on some subalgebra $\B^*_\alpha$ of $\B_\alpha$) which is given in advance (by a proper bookkeeping), cf. Theorem \ref{main_CH}.
    \item [(F)] We define a finite support iteration $(\PP_\alpha)_{\alpha\leqslant\omega_1}$ of $\sigma$-centered forcings (Definition \ref{iteration-definition}). For every $\alpha<\omega_1$ the algebra $\B_\alpha$ belongs to the $\alpha$-th intermediate model obtained from this iteration. The triples $(\B_\alpha, \B_{\alpha+1}, (\nu_n)_{n\in\N})$ satisfy $(\mathcal{G}^*)$ for uncountably many sequences of measures (on $\B_\alpha$), whose choice depends on a generic filter in $\PP_\alpha$ (for the connection between the choice of sequences of measures and a generic filter see Lemma \ref{one-step-forcing}).  
\end{itemize}
We finish by taking the Boolean algebra $$\B=\bigcup_{\alpha<\omega_1} \B_\alpha,$$ which is balanced and satisfies $(\mathcal{G})$.

    \item We start with $\B_0=\clop(C)$. At limit steps we take unions. The only non-trivial step is the construction of $\B_{\alpha+1}$ from $\B_\alpha$. In this case we extend $\B_\alpha$ by a new set $G\in\bor(C)$ which is a union of countably many pairwise disjoint elements of $\B_\alpha$: $$G=\bigcup_{n\in\N} G_n$$ that satisfy the hypothesis of Lemma \ref{NN-lemma} (this ensures that $\B_{\alpha+1}$ is balanced). Moreover, we require that
    \begin{itemize}
    \item [({\sf CH})] $G$ (together with some antichain $\{H_0^n, H_1^n\}_{n\in\N}\subseteq \B_\alpha)$ is a witness for the property $(\mathcal{G}^*)$ for the triple $(\B^*_\alpha, \B_{\alpha+1}, (\nu_n^\alpha)_{n\in\N})$,

    \item [(F)] there is an antichain $\{H_n\}_{n\in\N}\subseteq \B_\alpha$ such that for every sequence $(\nu_n)_{n\in\N}$ satisfying the hypothesis of Proposition \ref{one-step-thm} the set $G$ together with some subset of $\{H_n\}_{n\in\N}$ witnesses the property $(\mathcal{G}^*)$ for $(\B_\alpha, \B_{\alpha+1}, (\nu_n)_{n\in\N})$.
    \end{itemize}
\item From now on we will assume that $\alpha$ is fixed and we will focus on the construction of $G_n$'s, $H_n$'s and $H_i^n$'s for $i=0,1$. 

\begin{itemize}
    \item [({\sf CH})] We define $(G_n)_{n\in\N}, (H_0^n)_{n\in\N}, (H_1^n)_{n\in\N}$ by induction on $n\in \N$ (cf. Lemma \ref{one_step_extension}). In order to obtain the property $(\mathcal{G}^*)$ we need to ensure that for every $n\in \N$ the set $G_n\cap H_0^n$ is ``big'' in the sense of some measure from the sequence $(\nu_n^\alpha)_{n\in\N}$ while $G_n\cap H_1^k=\varnothing$ for $k,n\in \N$.
    \item [(F)] The sets $G_n$'s and $H_n$'s appear in the forcing conditions chosen by a generic filter. Lemma \ref{one-step-forcing} will imply that for an appropriate sequence $(\nu_n)_{n\in\N}$ of measures, for infinitely many $n\in\N$ the set $G_n\cap H_n$ is ``big'' in the sense of some measure from this sequence, while $G_n\cap H_k=\varnothing$ for every $n\in\N$ and infinitely many $k\in\N$. 
\end{itemize}

\item  Given finite sequences $(G_k)_{k\leqslant n}, (H_0^k)_{k\leqslant n},(H_1^k)_{k\leqslant n}$ (or $(H_k)_{k\leqslant n}$ in the case of forcing),  we extend them using Lemma \ref{2_elements_from_antichain} and Lemma \ref{construction-finite-step} (applied to $\widehat{G}= \bigcup_{k\leqslant n} G_k$ and $\widehat{H}= \bigcup_{k\leqslant n} (H_0^k\cup H_1^k)$ or $\widehat{H}= \bigcup_{k\leqslant n} H_k$).
\begin{itemize}
    \item [({\sf CH})] The set $G_{n+1}$ consists of 2 parts: $G_{n+1}=L\cup M$, where $L=G_{n+1}\cap H_0^{n+1}$ is the part witnessing the property $(\mathcal{G}^*)$ and $M$ is a very small set disjoint from $H_i^k$'s (so it has no influence on $(\mathcal{G}^*)$).
    \item [(F)] It follows that there are $l_1, l_2>n$ such that $G_{l_1}=L\cup M$, where $L=G_{l_1}\cap H_{l_1}$ witnesses the property $(\mathcal{G}^*)$ and $M\cap H_k=\varnothing$ for $k\in\N$, while $G_{l_2}\cap H_{l_2}=\varnothing$ (this will ensure that $G\cap H_{l_2}=\varnothing$).      
\end{itemize}
\item At the same time, to make sure that the hypothesis of Lemma \ref{NN-lemma} is satisfied, we need to pick $G_{n}$'s in such a way that the families $\FF(\B_n, \bigcup_{i\leqslant k} G_i)$ (where $\B_n$'s are some finite subalgebras of $\B_\alpha$) will have appropriate degrees of balance for $k,n\in\N$. However, we do not have enough control over the choice of $L$, which can affect the balance. Thus, we need to fix it with the help of $M$.

\item The choice of $M$ depends on $L$ in the way described in Proposition \ref{poprawka}. The main idea behind this choice is to reduce the problem to finite combinatorics. We work with some finite subalgebra $\HH$ of $\B$ containing the sets that have appeared in the construction so far (including $L$) that is sufficiently well balanced (cf. Lemma \ref{tristane}). 
\item Lemma \ref{homomorphism} shows that there is $n\in\N$ and a Boolean homomorphism $h\colon \HH \to \A_n$ (recall that $\A_n$ is a finite Boolean algebra consisting of clopen subsets of $C$) such that every $A\in \HH$ is well-approximated by $h(A)$. The choice of $M$ and most of the crucial calculations take place in $\A_n$ (see Lemma \ref{lemme}). These include the use of techniques such as probability inequalities involving weighted Rademacher sums (Lemma \ref{wcale_nie_Paley_Zygmund}) and analysis in finite-dimensional subspaces of the Hilbert space $\mathcal{L}_2(C)$.
\end{enumerate}

\section{Properties of balanced families}\label{balanced-families-section}
This section is devoted to the combinatorics of balanced sets and families. 
In a series of lemmas we will describe basic properties of balanced sets and show how to modify a given set to a balanced one. 

We start with a few simple observations. 

\begin{lemma}\label{disjoint_union}
    If $A,B\in \bor(C)$ are disjoint and $(m,\varepsilon)$-balanced, then $A\cup B$ is $(m,2\varepsilon)$-balanced.
\end{lemma}
\begin{proof}
First we check (\ref{3.1.1}).
Fix any $s\in \{-1,1\}^m$. If $$\frac{\lambda(\downset{s}\backslash A)}{\lambda(\downset{s})}< \frac{\varepsilon}{m} \ \text{or} \ \frac{\lambda(\downset{s}\backslash B)}{\lambda(\downset{s})}< \frac{\varepsilon}{m}$$ 
then 
$$\frac{\lambda(\downset{s}\backslash (A\cup B))}{\lambda(\downset{s})}\leqslant \min\left\{\frac{\lambda(\downset{s}\backslash A)}{\lambda(\downset{s})}, \frac{\lambda(\downset{s}\backslash B)}{\lambda(\downset{s})}\right\}< \frac{\varepsilon}{m}.$$
Otherwise, since $A$ and $B$ satisfy (\ref{3.1.1}) we have  
$$ \frac{\lambda(A\cap \downset{s})}{\lambda(\downset{s})}< \frac{\varepsilon}{m} \ \text{and} \ \frac{\lambda(B\cap \downset{s})}{\lambda(\downset{s})}< \frac{\varepsilon}{m}. $$
Summing up the inequalities we get 
$$\frac{\lambda((A\cup B) \cap \downset{s})}{\lambda(\downset{s})}< \frac{2\varepsilon}{m},$$
so $A\cup B$ satisfies (\ref{3.1.1}) for $2\varepsilon$.

Now we check (\ref{3.1.3}). Fix $r>m$. Since $A\cap B=\varnothing$ we have $$|\varphi_r((A\cup B)\cap \downset{s})|\leqslant |\varphi_r(A\cap \downset{s})|+|\varphi_r(B\cap \downset{s})|.$$ Thus, by summing up the inequalities (\ref{3.1.3}) for $A$ and $B$ we get 
$$\frac{|\varphi_r((A\cup B)\cap \downset{s})|}{\lambda(\downset{s})} \leqslant \frac{|\varphi_r(A\cap \downset{s})|}{\lambda(\downset{s})} + \frac{|\varphi_r(B\cap \downset{s})|}{\lambda(\downset{s})} < \frac{2\varepsilon}{r},$$
so $A\cup B$ satisfies (\ref{3.1.3}) for $2\varepsilon$. 
\end{proof}

\begin{lemma}\label{krwawy_baron}
    Let $A\in \bor(C), t\in \N$ and $\varepsilon>0$. If $A$ is $(t,\varepsilon)$-semibalanced then $C\backslash A$ is also $(t,\varepsilon)$-semibalanced.
\end{lemma}
\begin{proof}
    Let $r>t$. We have $$\varphi_r(A)+\varphi_r(C\backslash A)=\int_A \delta_r d\lambda + \int_{C\backslash A} \delta_r d\lambda = \int_C \delta_r d\lambda = 0.$$
    By (\ref{3.0.1}) for $A$ we get 
    $$|\varphi(C\backslash A)|= |\varphi(A)| <\frac{\varepsilon}{r},$$
    which implies (\ref{3.0.1}) for $C\backslash A$ and so $C\backslash A$ is $(t,\varepsilon)$-semibalanced.
\end{proof}

\begin{lemma}\label{algebra_E}
Suppose that $\F$ is a finite Boolean subalgebra of $\bor(C)$ and $t\in\N$ is such that $\F$ is $(t,\varepsilon)$-balanced, 
where $$ \varepsilon= \frac{1}{100}\inf \{\lambda(A): A\in \F, \lambda(A)>0 \}.$$
Then for every $A\in\F$, if $\lambda(A)>0$, then there is $s_A\in \{-1,1\}^t$ such that 
$$\frac{\lambda(A\cap \downset{s_A})}{\lambda(\downset{s_A})}\geqslant 0.99.$$
\end{lemma}
\begin{proof}
    Let $A\in \F$ be such that $\lambda(A)>0$. Then $\lambda(A)>\varepsilon$, so there must be $s_A\in \{-1,1\}^t$ such that $\lambda(A\cap \downset{s_A})>\varepsilon \lambda(\downset{s_A})$. Hence
$$\frac{\lambda(\downset{s_A}\backslash A)}{\lambda(\downset{s_A})}\leqslant 1- \varepsilon.$$
Since $A$ satisfies (\ref{3.1.1}) we have  
$$\frac{\lambda(\downset{s_A}\backslash A)}{\lambda(\downset{s_A})}<\frac{\varepsilon}{t}\leqslant \varepsilon.$$
But $\varepsilon\leqslant 0.01$, so 
$$\frac{\lambda(A\cap \downset{s_A})}{\lambda(\downset{s_A})}=\frac{\lambda(\downset{s_A})-\lambda(\downset{s_A}\backslash A)}{\lambda(\downset{s_A})} >1-\varepsilon \geqslant 0.99.$$
\end{proof}

\begin{lemma}\label{miecze_i_pierogi}
    Let $\HH_0$ be an $(n, \varepsilon)$-balanced Boolean algebra. Then the Boolean algebra $\HH$ generated by $\HH_0\cup \A_n$ is $(n, \varepsilon)$-balanced.
\end{lemma}

\begin{proof}
    Take $A \in \HH$. Then $A$ is of the form:
    
    $$A=\bigcup_{\downset{s}\in \at(\A_n)} (\downset{s} \cap   A_s)=\bigcup_{s\in \{-1,1\}^n} (\downset{s} \cap A_s),$$ 
where $A_s \in \HH_0$ for $s\in \{-1,1\}^n$.

Let $s_0\in \{-1,1\}^n$. Then 

$$A\cap \downset{s_0} = \bigcup_{s\in \{-1,1\}^n}(\downset{s} \cap A_s)\cap \downset{s_0}=A_{s_0} \cap \downset{s_0}$$

and

$$\downset{s_0}\backslash A =\downset{s_0}\backslash \bigcup_{s\in \{-1,1\}^n}(\downset{s} \cap A_s)=\downset{s_0}\backslash A_{s_0}.$$

By (\ref{3.1.1}) we have $$\frac{\lambda(A_{s_0} \cap \downset{s_0})}{\lambda(\downset{s_0})}< \frac{\varepsilon}{n} 
 \ \text{or} \ 
\frac{\lambda(\downset{s_0}\backslash A_{s_0})}{\lambda(\downset{s_0})}< \frac{\varepsilon}{n}.$$

If $\frac{\lambda(A_{s_0} \cap \downset{s_0})}{\lambda(\downset{s_0})}< \frac{\varepsilon}{n}$, then

$$\frac{\lambda(A\cap \downset{s_0})}{\lambda(\downset{s_0})}=\frac{\lambda(A_{s_0} \cap \downset{s_0})}{\lambda(\downset{s_0})}< \frac{\varepsilon}{n}.$$

If $\frac{\lambda(\downset{s_0}\backslash A_{s_0})}{\lambda(\downset{s_0})}< \frac{\varepsilon}{n}$, then

$$\frac{\lambda(\downset{s_0}\backslash A)}{\lambda(\downset{s_0})}=\frac{\lambda(\downset{s_0}\backslash A_{s_0})}{\lambda(\downset{s_0})}< \frac{\varepsilon}{n},$$
so $A$ also satisfies (\ref{3.1.1}).

    Let $r>n$ and let $s_0\in \{-1,1\}^n$. Then by (\ref{3.1.3})

    $$
    \frac{|\varphi_r(A\cap \downset{s_0})|}{\lambda(\downset{s_0})}=\frac{|\varphi_r(A_{s_0} \cap \downset{s_0})|}{\lambda(\downset{s_0})}
    < \frac{\varepsilon}{r}.
    $$
\end{proof}

The next lemma shows that while dealing with finite balanced families we can approximate them with finite families of clopen subsets of $C$. This will allow us to reduce many problems to the combinatorics of finite Boolean algebras $\A_n$ for $n\in \N$.

\begin{lemma}\label{homomorphism}
    Suppose that $\HH \subseteq \bor(C)$ is a finite subalgebra that is $(n, \varepsilon)$-balanced for some $n\in \N$ and $\varepsilon<1/3$. Then the function $h_n\colon \HH \rightarrow \A_n$ given by 
    $$h_n(A)= \bigcup \left\{\downset{s}: {s\in \{-1,1\}^n}, \frac{\lambda(\downset{s}\backslash A)}{\lambda(\downset{s})}<\varepsilon\right\}$$
    is a homomorphism of Boolean algebras and for every $A\in \HH$ we have $$\lambda\left(A\triangle h_n(A)\right)<\varepsilon/n.$$ 
\end{lemma}

\begin{proof}
   For the first part of the lemma we need to show that $h_n(C)=C, h_n(A\cup B)= h_n(A)\cup h_n(B)$ and $h_n(C\backslash A)=C\backslash h_n(A)$ for every $A,B\in \HH$. 

The first equality holds since for every $s\in \{-1,1\}^n$ we have $$\frac{\lambda(\downset{s}\backslash C)}{\lambda(\downset{s})}= \frac{\lambda(\varnothing)}{\lambda(\downset{s})}=0.$$
The second equality follows from the fact that 
   $$\frac{\lambda(\downset{s}\backslash (A\cup B))}{\lambda(\downset{s})}<\varepsilon \ \text{iff} \ \frac{\lambda(\downset{s}\backslash A)}{\lambda(\downset{s})}<\varepsilon \ \text{or} \ \frac{\lambda(\downset{s}\backslash B)}{\lambda(\downset{s})}<\varepsilon.$$
Indeed, if 
$$\frac{\lambda(\downset{s}\backslash (A\cup B))}{\lambda(\downset{s})}<\varepsilon,$$
   then $$\frac{\lambda((A\cup B)\cap \downset{s})}{\lambda(\downset{s})}\geqslant\frac{2}{3}$$
   and so
   $$\max \left\{\frac{\lambda(A\cap \downset{s})}{\lambda(\downset{s})}, \frac{\lambda(B\cap \downset{s})}{\lambda(\downset{s})} \right\}\geqslant \frac{1}{3} > \varepsilon. $$
   Since $A$ and $B$ are $(n, \varepsilon)$-balanced we have
   $$\max \left\{\frac{\lambda(A\cap \downset{s})}{\lambda(\downset{s})}, \frac{\lambda(B\cap \downset{s})}{\lambda(\downset{s})} \right\}> 1-\varepsilon $$
   or equivalently 
   $$\min \left\{ \frac{\lambda(\downset{s}\backslash A)}{\lambda(\downset{s})}, \frac{\lambda(\downset{s}\backslash B)}{\lambda(\downset{s})} \right\}< \varepsilon. $$
   Conversely, if $$\frac{\lambda(\downset{s}\backslash A)}{\lambda(\downset{s})}<\varepsilon \ \text{or} \ \frac{\lambda(\downset{s}\backslash B)}{\lambda(\downset{s})}<\varepsilon,$$ 
then $$\frac{\lambda(\downset{s}\backslash (A\cup B)}{\lambda(\downset{s})}\leqslant \min\left\{\frac{\lambda(\downset{s}\backslash A)}{\lambda(\downset{s})}, \frac{\lambda(\downset{s}\backslash B)}{\lambda(\downset{s})}\right\}<\varepsilon.$$
The equality $h_n(C\backslash A)=C\backslash h_n(A)$ holds since for $s\in \{-1,1\}^n$
\begin{gather*}
    \downset{s}\subseteq C\backslash h_n(A) \ \text{iff} \ \frac{\lambda(\downset{s}\backslash A)}{\lambda(\downset{s})}\geqslant\varepsilon \ \text{iff} \ \frac{\lambda(A\cap\downset{s})}{\lambda(\downset{s})}\leqslant 1-\varepsilon \ \text{iff} \\ 
    \frac{\lambda(A\cap\downset{s})}{\lambda(\downset{s})}<\varepsilon \ \text{iff} \ \frac{\lambda(\downset{s}\backslash (C\backslash A))}{\lambda(\downset{s})}<\varepsilon \ \text{iff} \ \downset{s}\subseteq h_n(C\backslash A).
\end{gather*}

   For the second part of the lemma we notice that for every $s\in \{-1,1\}^n$ if $\frac{\lambda(\downset{s}\backslash A)}{\lambda(\downset{s})} <\frac{\varepsilon}{n}$, then

   \begin{gather*}
    \lambda(A\triangle h_n(A) \cap \downset{s}) = \lambda(\downset{s}\backslash A)< \lambda(\downset{s})\frac{\varepsilon}{n}  
   \end{gather*}
and if $\frac{\lambda(\downset{s}\backslash A)}{\lambda(\downset{s})}\geqslant \frac{\varepsilon}{n}$, then $\frac{\lambda(A\cap \downset{s})}{\lambda(\downset{s})} <\frac{\varepsilon}{n}$, and so
\begin{gather*}
    \lambda(A\triangle h_n(A) \cap \downset{s})= \lambda(A\cap \downset{s})< \lambda(\downset{s})\frac{\varepsilon}{n}.
\end{gather*}
Hence
   \begin{gather*}
       \lambda(A\triangle h_n(A)) = \sum_{s\in \{-1,1\}^n} \lambda(A\triangle h_n(A) \cap \downset{s}) \leqslant \sum_{s\in \{-1,1\}^n} \lambda(\downset{s})\frac{\varepsilon}{n}= \frac{\varepsilon}{n}.
   \end{gather*}
\end{proof}

The next lemma says that small perturbations of $(m,t,\varepsilon)$-balanced sets are still  $(m,t,\varepsilon)$-balanced. 

\begin{lemma}\label{small_perturbation}
Let $m, t\in \N, t>m, \varepsilon >0$. Then there is $\mem>0$ such that for every $A, B\in \bor(C)$, if $A$ is $(m, t, \varepsilon)$-balanced and $\lambda(B)<\mem$, then $A\cup B$ and $A\backslash B$ are $(m,t, \varepsilon)$-balanced.
\end{lemma}

\begin{proof}
   Let $\varepsilon_1<\varepsilon$ be such that $A$ is $(m, t, \varepsilon_1)$-balanced and let 
   $$ \mem = \frac{\varepsilon-\varepsilon_1}{2^mt}.$$
   For every $s\in \{-1, 1\}^m$  we have 
\begin{gather*}
    \frac{\lambda((A\cup B)\cap \downset{s})}{\lambda(\downset{s})}\leqslant \frac{\lambda(A\cap \downset{s})}{\lambda(\downset{s})} + \frac{\mem}{\lambda(\downset{s})} < \frac{\varepsilon_1}{m} + \frac{\varepsilon - \varepsilon_1}{t} \leqslant \frac{\varepsilon}{m}
\end{gather*}
or 
\begin{gather*}
    \frac{\lambda(\downset{s}\backslash (A\cup B))}{\lambda(\downset{s})}\leqslant \frac{\lambda(\downset{s}\backslash A)}{\lambda(\downset{s})}< \frac{\varepsilon}{m}
\end{gather*}
and for every $s\in \{-1, 1\}^m$, $m<r\leqslant t$ 
\begin{gather*}
    \frac{|\varphi_r((A\cup B)\cap \downset{s})|}{\lambda(\downset{s})} \leqslant \frac{|\varphi_r(A\cap \downset{s})|}{\lambda(\downset{s})} + \frac{|\varphi_r((B\backslash A)\cap \downset{s})|}{\lambda(\downset{s})}\leqslant \frac{\varepsilon_1}{r}+\frac{\varepsilon-\varepsilon_1}{t} < \frac{\varepsilon}{r}.
\end{gather*}
Hence $A\cup B$ is  $(m, t, \varepsilon)$-balanced. Calculations showing that $A\backslash B$ is  $(m, t, \varepsilon)$-balanced are similar.
\end{proof}

In the following lemma, we provide conditions for enlarging balanced Boolean algebras to bigger ones. 

\begin{lemma}\label{NN-lemma}
    Suppose that $\B\subseteq \bor(C)$ is a balanced Boolean algebra and that $$\B=\bigcup_{n\in\N} \B_n$$ is a representation of $\B$ as an increasing union of finite subalgebras. Let $(m_n)_{n\in\N}$ be a strictly increasing sequence of natural numbers and $\{G_n\}_{n\in \N}\subseteq \B$ be an antichain such that 

$$ \forall k\in \N \ \forall n\leqslant k \ \FF\Big(\B_n, \bigcup_{i\leqslant k} G_i\Big) \ \text{is} \ (m_n, 2^{-n})\text{-balanced},$$
where $\FF\left(\B_n, \bigcup_{i\leqslant k} G_i\right)=\left\{A\cap \bigcup_{i\leqslant k} G_i, A\backslash \bigcup_{i\leqslant k} G_i: A\in \B_n\right\}$.

Put $G=\bigcup_{n\in \N} G_n$.
Then the Boolean algebra $\B'$ generated by $\B\cup\{G\}$ is balanced.
\end{lemma}

\begin{proof}
   Let $\varepsilon>0$ and let $\AAA'$ be a finite subfamily of $\B'$. Every element $A'\in \AAA'$ is of the form  $(A_1\cap G)\cup(A_2\backslash G)$ where $A_1, A_2\in \B$, so there is a finite family $\AAA\subseteq \B$ such that $\AAA'\subseteq \{(A_1\cap G)\cup(A_2\backslash G): A_1,A_2\in \AAA\}$. Let $n\in \N$ be such that $\AAA\subseteq \B_n$ and  $1/2^{n-1}<\varepsilon$. Fix $A_1, A_2\in \AAA$. Then for every $k\geqslant n$  
   $$A_1\cap \bigcup_{i\leqslant k} G_i \ \text{and} \ A_2\backslash \bigcup_{i\leqslant k} G_i \ \text{are} \ (m_n, 1/2^n)-\text{balanced}.$$ 
    Hence for $s\in \{-1,1\}^{m_n}$ we have
    \begin{eqnarray*}
         \frac{\lambda(A_1 \cap G \cap  \downset{s})}{\lambda(\downset{s})}&\leqslant& \frac{\lambda\left(A_1 \cap \bigcup_{i\leqslant k} G_i \cap  \downset{s}\right)+ \lambda\left(\bigcup_{i> n} G_i\right)}{\lambda(\downset{s})}\leqslant \\ &\leqslant& \frac{1}{m_n2^{n}} + \frac{\lambda\left(\bigcup_{i> k} G_i\right)}{\lambda(\downset{s})} \xrightarrow{k\rightarrow \infty} \frac{1}{m_n2^{n}}  
    \end{eqnarray*}
    or
    \begin{eqnarray*}  
         \frac{\lambda(\downset{s}\backslash (A_1 \cap G))}{\lambda(\downset{s})}&\leqslant& \frac{\lambda\left(\downset{s}\backslash \left(A_1 \cap \bigcup_{i\leqslant k} G_i\right)\right)}{\lambda(\downset{s})}\leqslant  \frac{1}{m_n2^{n}} 
    \end{eqnarray*}
    and for $m> m_n$
    \begin{eqnarray*}
        \frac{|\varphi_m(A_1\cap G\cap \downset{s})|}{\lambda(\downset{s})} \!\!\!\!\! &\leqslant\!\!\!\!\!&  \frac{\left|\varphi_m\left(A_1\cap \bigcup_{i\leqslant k} G_i\cap \downset{s}\right)\right|}{\lambda(\downset{s})} + \frac{\left|\varphi_m\Big(A_1\cap \bigcup_{i> k} G_i\cap \downset{s}\Big)\right|}{\lambda(\downset{s})} \leqslant \\ \!\!\!\!\!&\leqslant\!\!\!\!\!& \frac{1}{m2^n} +\frac{\lambda\left(\bigcup_{i> k} G_i\right)}{\lambda(\downset{s})} \xrightarrow{k \rightarrow\infty}  \frac{1}{m2^n},  
    \end{eqnarray*}
    so $A_1\cap G$ is $(m_n, 1/2^n)$-balanced. By a similar argument $A_2\backslash G$ is $(m_n, 1/2^n)$-balanced. By Lemma \ref{disjoint_union} the set $(A_1\cap G)\cup (A_2\backslash G)$ is $(m_n, 1/2^{n-1})$-balanced, and so  $(m_n,\varepsilon)$-balanced, which completes the proof. 
\end{proof}

The next proposition is key for the construction in Lemma \ref{construction-finite-step}. It says how much we can modify a given finite balanced family without losing its balance.
 
\begin{proposition}\label{poprawka}
   Let $k\in \N$, $\eta >0$. Let $(m_n)_{n\leqslant k}$ be an increasing sequence of natural numbers. Let $\B^*\subseteq\B\subseteq \bor(C)$ be balanced Boolean algebras and assume that $\clop(C)\subseteq \B^*$. Let $(\B_n)_{n\leqslant k}\subseteq \B$ be finite subalgebras. Suppose that $G, P\in \B^*$ and the following are satisfied:
   \begin{enumerate}
       \item [(A)] $G\subseteq P$,
       \item [(B)] $\forall n\leqslant k \ \FF(\B_n, G)$ is $(m_n, 2^{-n})$-balanced. 
   \end{enumerate}
Then there is $\theta>0$ such that for every $L,Q\in \B^*$ satisfying 
\begin{enumerate}
    \item [(a)] $\max\{\lambda (L), \lambda(Q)\}<\theta$,
    \item [(b)] $L\cap P=\varnothing$,
\end{enumerate}
there is $M\in \B^*$ such that 
 \begin{enumerate}
     \item  $M\cap (P\cup Q)=\varnothing$,
     \item  $\lambda(M)<\eta$,
     \item  $\forall n\leqslant k \ \FF(\B_n, G\cup L\cup M)$ is $(m_n, 2^{-n})$-balanced. 
 \end{enumerate}
   
\end{proposition}
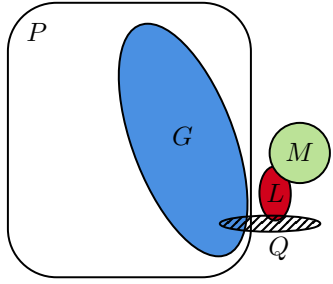
\begin{center}
 
\tikzset{
pattern size/.store in=\mcSize, 
pattern size = 5pt,
pattern thickness/.store in=\mcThickness, 
pattern thickness = 0.3pt,
pattern radius/.store in=\mcRadius, 
pattern radius = 1pt}
\makeatletter
\pgfutil@ifundefined{pgf@pattern@name@_2zf3f2suk}{
\pgfdeclarepatternformonly[\mcThickness,\mcSize]{_2zf3f2suk}
{\pgfqpoint{0pt}{0pt}}
{\pgfpoint{\mcSize+\mcThickness}{\mcSize+\mcThickness}}
{\pgfpoint{\mcSize}{\mcSize}}
{
\pgfsetcolor{\tikz@pattern@color}
\pgfsetlinewidth{\mcThickness}
\pgfpathmoveto{\pgfqpoint{0pt}{0pt}}
\pgfpathlineto{\pgfpoint{\mcSize+\mcThickness}{\mcSize+\mcThickness}}
\pgfusepath{stroke}
}}
\makeatother
\tikzset{every picture/.style={line width=0.75pt}} 

\begin{tikzpicture}[x=0.75pt,y=0.75pt,yscale=-1,xscale=1]

\draw  [fill={rgb, 255:red, 74; green, 144; blue, 226 }  ,fill opacity=1 ] (251.28,444) .. controls (241.42,411.58) and (245.55,385.3) .. (260.52,385.3) .. controls (275.49,385.3) and (295.62,411.58) .. (305.48,444) .. controls (315.34,476.42) and (311.2,502.7) .. (296.23,502.7) .. controls (281.27,502.7) and (261.14,476.42) .. (251.28,444) -- cycle ;
\draw   (190,399.24) .. controls (190,385.7) and (200.97,374.73) .. (214.51,374.73) -- (288.05,374.73) .. controls (301.58,374.73) and (312.56,385.7) .. (312.56,399.24) -- (312.56,488.76) .. controls (312.56,502.3) and (301.58,513.27) .. (288.05,513.27) -- (214.51,513.27) .. controls (200.97,513.27) and (190,502.3) .. (190,488.76) -- cycle ;
\draw  [fill={rgb, 255:red, 208; green, 2; blue, 27 }  ,fill opacity=0.86 ] (324.51,457.3) .. controls (328.87,457.22) and (332.52,463.27) .. (332.66,470.81) .. controls (332.81,478.35) and (329.39,484.53) .. (325.02,484.61) .. controls (320.66,484.69) and (317.01,478.65) .. (316.87,471.11) .. controls (316.73,463.56) and (320.15,457.38) .. (324.51,457.3) -- cycle ;
\draw  [fill={rgb, 255:red, 187; green, 226; blue, 151 }  ,fill opacity=0.75 ] (322,450.54) .. controls (322,442.14) and (328.81,435.33) .. (337.2,435.33) .. controls (345.6,435.33) and (352.41,442.14) .. (352.41,450.54) .. controls (352.41,458.93) and (345.6,465.74) .. (337.2,465.74) .. controls (328.81,465.74) and (322,458.93) .. (322,450.54) -- cycle ;
\draw  [pattern=_2zf3f2suk,pattern size=3.2249999999999996pt,pattern thickness=0.75pt,pattern radius=0pt, pattern color={rgb, 255:red, 0; green, 0; blue, 0}] (296.88,486.22) .. controls (296.88,483.98) and (308.23,482.16) .. (322.22,482.16) .. controls (336.21,482.16) and (347.56,483.98) .. (347.56,486.22) .. controls (347.56,488.46) and (336.21,490.27) .. (322.22,490.27) .. controls (308.23,490.27) and (296.88,488.46) .. (296.88,486.22) -- cycle ;

\draw (271.33,435.67) node [anchor=north west][inner sep=0.75pt]    {$G$};
\draw (328.5,444) node [anchor=north west][inner sep=0.75pt]  [color={rgb, 255:red, 0; green, 0; blue, 0 }  ,opacity=1 ]  {$M$};
\draw (319,464.5) node [anchor=north west][inner sep=0.75pt]    {$L$};
\draw (320,491.27) node [anchor=north west][inner sep=0.75pt]    {$Q$};
\draw (197.94,383.18) node [anchor=north west][inner sep=0.75pt]    {$P$};

\end{tikzpicture}
\end{center}
Before we prove the proposition, we will need a few lemmas.

The following lemma is a version of \cite[Theorem 3, page 31]{kahane}.
\begin{lemma}\label{wcale_nie_Paley_Zygmund}
   Let $n\in\N$. Let $\lambda_n$ be the standard product probability measure on the space $\{-1, 1\}^n$ ($\lambda(\{x\})=\lambda(\{y\})$ for every $x,y\in \{-1, 1\}^n$). Then for all $(d_m)_{m\leqslant n}\in \R^n$ and any $\xi\in(0, 1)$

    $$
    \lambda_n \left(\left\{y\in\{-1, 1\}^n: \left|\sum_{m=1}^n y_m d_m\right|^2\geqslant\xi \sum_{m=1}^n |d_m|^2\right\}\right) \geqslant\frac{1}{3} (1-\xi)^2.
    $$
\end{lemma}

The above lemma will allow us to pick $y$ from some big enough subset of $\{-1, 1\}^n$ satisfying the appropriate inequality.

\begin{lemma}\label{orthonormal}
    The sequence $(\delta_n)_{n\in\N}$ is orthonormal in the Hilbert space $\mathcal{L}_2(C)$. 
\end{lemma}

\begin{proof}
    We need to show that for any $n,m\in \N$
    \begin{gather*}
           \downset{\delta_n,\delta_m} =
\left\{
	\begin{array}{ll}
		1  & \mbox{if }  n=m, \\
		0 & \mbox{if } n\neq m.
	\end{array}
\right.  
\end{gather*}
If $n=m$, then $\downset{\delta_n,\delta_m}= \|\delta_n\|=1$. 

Suppose that $n\neq m$. For $i,j\in \{-1,1\}$ let $C^i_j = \{s\in C: s_n=i, s_m=j\}$ and note that if $x\in C^i_j$, then 
\begin{gather*}
           \delta_n(x)\delta_m(x) = ij.
\end{gather*}
Since $\lambda(C_j^i)=1/4$ for $i,j\in\{-1,1\}$, we get 
\begin{gather*}
    \downset{\delta_n,\delta_m} = \int_C \delta_n\delta_m d\lambda = \sum_{i,j\in \{-1,1\}} \int_{C^i_j} \delta_n\delta_m d\lambda = \sum_{i,j\in \{-1,1\}} ij\lambda(C^i_j)=0.
\end{gather*}

\end{proof}

\begin{lemma}\label{lemme}
    Let $t\in \N$. Let $\eta \in (0, 1/2^{t+10})$. Let $n_0$ be large enough so that for all $n>n_0$ there exists $k\in \N$ that 
    $$\frac{\eta}{2} <\frac{k}{2^n}< \eta \ \text{and} \ \frac{n^3}{2^{n-1}}\leqslant \eta. $$ Let $n>n_0$ and $k$ satisfy the above inequality. Then for all
    \begin{itemize}
        \item  $Q \in \A_n$ such that $\lambda(Q)<\eta$,
       
        \item  $F\in \A_n$ for which there exists $s \in \{-1, 1\}^t$ such that $\frac{\lambda(F \cap \downset{s})}{\lambda(\downset{s})}\geqslant 0.95,$

         \item  $Z \in \A_n$ such that $\lambda(Z)<\frac{\eta^2}{64}$ and $Z \subseteq F$,
    \end{itemize}
    there exists $M \subseteq F \setminus (Q \cup Z)$ such that
    \begin{enumerate}
        \item [(a)] $M \in \A_n$,
        \item [(b)] $\lambda(M)=\frac{k}{2^n}$,
        \item [(c)] $M\cup Z$ is $(t,\eta)$-semibalanced.
    \end{enumerate}

\end{lemma}

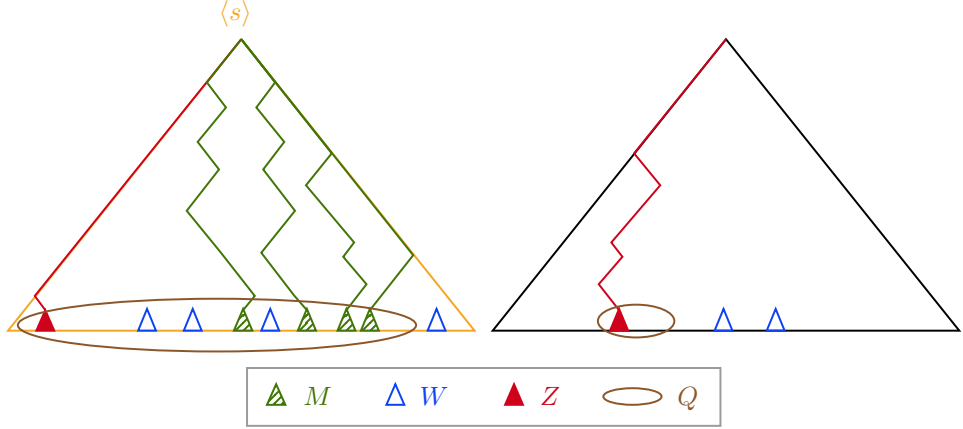
\begin{center}
 
 
\tikzset{
pattern size/.store in=\mcSize, 
pattern size = 5pt,
pattern thickness/.store in=\mcThickness, 
pattern thickness = 0.3pt,
pattern radius/.store in=\mcRadius, 
pattern radius = 1pt}
\makeatletter
\pgfutil@ifundefined{pgf@pattern@name@_tlo87qvif}{
\pgfdeclarepatternformonly[\mcThickness,\mcSize]{_tlo87qvif}
{\pgfqpoint{0pt}{0pt}}
{\pgfpoint{\mcSize+\mcThickness}{\mcSize+\mcThickness}}
{\pgfpoint{\mcSize}{\mcSize}}
{
\pgfsetcolor{\tikz@pattern@color}
\pgfsetlinewidth{\mcThickness}
\pgfpathmoveto{\pgfqpoint{0pt}{0pt}}
\pgfpathlineto{\pgfpoint{\mcSize+\mcThickness}{\mcSize+\mcThickness}}
\pgfusepath{stroke}
}}
\makeatother

 
\tikzset{
pattern size/.store in=\mcSize, 
pattern size = 5pt,
pattern thickness/.store in=\mcThickness, 
pattern thickness = 0.3pt,
pattern radius/.store in=\mcRadius, 
pattern radius = 1pt}
\makeatletter
\pgfutil@ifundefined{pgf@pattern@name@_xk1y1x3w1}{
\pgfdeclarepatternformonly[\mcThickness,\mcSize]{_xk1y1x3w1}
{\pgfqpoint{0pt}{0pt}}
{\pgfpoint{\mcSize+\mcThickness}{\mcSize+\mcThickness}}
{\pgfpoint{\mcSize}{\mcSize}}
{
\pgfsetcolor{\tikz@pattern@color}
\pgfsetlinewidth{\mcThickness}
\pgfpathmoveto{\pgfqpoint{0pt}{0pt}}
\pgfpathlineto{\pgfpoint{\mcSize+\mcThickness}{\mcSize+\mcThickness}}
\pgfusepath{stroke}
}}
\makeatother

 
\tikzset{
pattern size/.store in=\mcSize, 
pattern size = 5pt,
pattern thickness/.store in=\mcThickness, 
pattern thickness = 0.3pt,
pattern radius/.store in=\mcRadius, 
pattern radius = 1pt}
\makeatletter
\pgfutil@ifundefined{pgf@pattern@name@_vci555xxo}{
\pgfdeclarepatternformonly[\mcThickness,\mcSize]{_vci555xxo}
{\pgfqpoint{0pt}{0pt}}
{\pgfpoint{\mcSize+\mcThickness}{\mcSize+\mcThickness}}
{\pgfpoint{\mcSize}{\mcSize}}
{
\pgfsetcolor{\tikz@pattern@color}
\pgfsetlinewidth{\mcThickness}
\pgfpathmoveto{\pgfqpoint{0pt}{0pt}}
\pgfpathlineto{\pgfpoint{\mcSize+\mcThickness}{\mcSize+\mcThickness}}
\pgfusepath{stroke}
}}
\makeatother

 
\tikzset{
pattern size/.store in=\mcSize, 
pattern size = 5pt,
pattern thickness/.store in=\mcThickness, 
pattern thickness = 0.3pt,
pattern radius/.store in=\mcRadius, 
pattern radius = 1pt}
\makeatletter
\pgfutil@ifundefined{pgf@pattern@name@_50kmkb5pa}{
\pgfdeclarepatternformonly[\mcThickness,\mcSize]{_50kmkb5pa}
{\pgfqpoint{0pt}{0pt}}
{\pgfpoint{\mcSize+\mcThickness}{\mcSize+\mcThickness}}
{\pgfpoint{\mcSize}{\mcSize}}
{
\pgfsetcolor{\tikz@pattern@color}
\pgfsetlinewidth{\mcThickness}
\pgfpathmoveto{\pgfqpoint{0pt}{0pt}}
\pgfpathlineto{\pgfpoint{\mcSize+\mcThickness}{\mcSize+\mcThickness}}
\pgfusepath{stroke}
}}
\makeatother

 
\tikzset{
pattern size/.store in=\mcSize, 
pattern size = 5pt,
pattern thickness/.store in=\mcThickness, 
pattern thickness = 0.3pt,
pattern radius/.store in=\mcRadius, 
pattern radius = 1pt}
\makeatletter
\pgfutil@ifundefined{pgf@pattern@name@_9g2mn5ju6}{
\pgfdeclarepatternformonly[\mcThickness,\mcSize]{_9g2mn5ju6}
{\pgfqpoint{0pt}{0pt}}
{\pgfpoint{\mcSize+\mcThickness}{\mcSize+\mcThickness}}
{\pgfpoint{\mcSize}{\mcSize}}
{
\pgfsetcolor{\tikz@pattern@color}
\pgfsetlinewidth{\mcThickness}
\pgfpathmoveto{\pgfqpoint{0pt}{0pt}}
\pgfpathlineto{\pgfpoint{\mcSize+\mcThickness}{\mcSize+\mcThickness}}
\pgfusepath{stroke}
}}
\makeatother
\tikzset{every picture/.style={line width=0.75pt}} 

\begin{tikzpicture}[x=0.75pt,y=0.75pt,yscale=-1,xscale=1]

\draw  [color={rgb, 255:red, 245; green, 166; blue, 35 }  ,draw opacity=1 ] (165.88,46.91) -- (283.54,194) -- (48.23,194) -- cycle ;
\draw  [color={rgb, 255:red, 208; green, 2; blue, 27 }  ,draw opacity=1 ][fill={rgb, 255:red, 208; green, 2; blue, 27 }  ,fill opacity=1 ] (67.05,183.65) -- (71.55,194) -- (62.54,194) -- cycle ;
\draw [color={rgb, 255:red, 208; green, 2; blue, 27 }  ,draw opacity=1 ]   (67.05,183.65) -- (61.75,176.46) -- (165.88,46.91) ;
\draw [color={rgb, 255:red, 65; green, 117; blue, 5 }  ,draw opacity=1 ]   (166.81,183.65) -- (172.69,176.46) -- (155.63,154.66) -- (138.38,133.49) -- (154.65,112.63) -- (143.89,98.73) -- (158.17,81.35) -- (148.55,68.71) -- (165.88,46.91) ;
\draw [color={rgb, 255:red, 65; green, 117; blue, 5 }  ,draw opacity=1 ]   (165.88,46.91) -- (252.57,155.92) -- (230.74,183.65) ;
\draw  [color={rgb, 255:red, 65; green, 117; blue, 5 }  ,draw opacity=1 ][pattern=_tlo87qvif,pattern size=3pt,pattern thickness=0.75pt,pattern radius=0pt, pattern color={rgb, 255:red, 65; green, 117; blue, 5}] (230.74,183) -- (235.25,194) -- (226.23,194) -- cycle ;
\draw [color={rgb, 255:red, 65; green, 117; blue, 5 }  ,draw opacity=1 ]   (218.93,183.65) -- (229.03,170.56) -- (217.3,156.71) -- (222.79,149.48) -- (198.52,120.61) -- (211.5,104.49) -- (165.88,46.91) ;
\draw  [color={rgb, 255:red, 65; green, 117; blue, 5 }  ,draw opacity=1 ][pattern=_xk1y1x3w1,pattern size=3pt,pattern thickness=0.75pt,pattern radius=0pt, pattern color={rgb, 255:red, 65; green, 117; blue, 5}] (218.93,183) -- (223.44,194) -- (214.42,194) -- cycle ;
\draw  [color={rgb, 255:red, 65; green, 117; blue, 5 }  ,draw opacity=1 ][pattern=_vci555xxo,pattern size=3pt,pattern thickness=0.75pt,pattern radius=0pt, pattern color={rgb, 255:red, 65; green, 117; blue, 5}] (166.81,183) -- (171.46,194) -- (162.16,194) -- cycle ;
\draw [color={rgb, 255:red, 65; green, 117; blue, 5 }  ,draw opacity=1 ]   (198.98,183.65) -- (192.47,176.15) -- (175.98,154.66) -- (192.97,133.49) -- (176.94,112.63) -- (187.55,98.73) -- (173.48,81.35) -- (182.95,68.71) -- (165.88,46.91) ;
\draw  [color={rgb, 255:red, 65; green, 117; blue, 5 }  ,draw opacity=1 ][pattern=_50kmkb5pa,pattern size=3pt,pattern thickness=0.75pt,pattern radius=0pt, pattern color={rgb, 255:red, 65; green, 117; blue, 5}] (198.98,183) -- (194.4,194) -- (203.56,194) -- cycle ;
\draw  [color={rgb, 255:red, 0; green, 0; blue, 0 }  ,draw opacity=1 ] (410.34,46.91) -- (528,194) -- (292.69,194) -- cycle ;
\draw [color={rgb, 255:red, 208; green, 2; blue, 27 }  ,draw opacity=1 ]   (356.43,183.65) -- (346.18,170.56) -- (358.1,156.71) -- (352.52,149.48) -- (377.18,120.61) -- (363.99,104.49) -- (410.34,46.91) ;
\draw  [color={rgb, 255:red, 139; green, 87; blue, 42 }  ,draw opacity=1 ] (345.74,189.15) .. controls (345.74,184.59) and (354.37,180.88) .. (365.01,180.88) .. controls (375.66,180.88) and (384.29,184.59) .. (384.29,189.15) .. controls (384.29,193.72) and (375.66,197.42) .. (365.01,197.42) .. controls (354.37,197.42) and (345.74,193.72) .. (345.74,189.15) -- cycle ;
\draw  [color={rgb, 255:red, 139; green, 87; blue, 42 }  ,draw opacity=1 ] (53.2,191.19) .. controls (53.2,183.84) and (98.16,177.88) .. (153.63,177.88) .. controls (209.09,177.88) and (254.05,183.84) .. (254.05,191.19) .. controls (254.05,198.54) and (209.09,204.5) .. (153.63,204.5) .. controls (98.16,204.5) and (53.2,198.54) .. (53.2,191.19) -- cycle ;
\draw  [color={rgb, 255:red, 208; green, 2; blue, 27 }  ,draw opacity=1 ][fill={rgb, 255:red, 208; green, 2; blue, 27 }  ,fill opacity=1 ] (356.43,183) -- (351.85,194) -- (361.01,194) -- cycle ;
\draw  [color={rgb, 255:red, 14; green, 67; blue, 243 }  ,draw opacity=1 ] (409.1,183) -- (413.6,194) -- (404.59,194) -- cycle ;
\draw  [color={rgb, 255:red, 14; green, 67; blue, 243 }  ,draw opacity=1 ] (264.37,183) -- (268.88,194) -- (259.87,194) -- cycle ;
\draw  [color={rgb, 255:red, 14; green, 67; blue, 243 }  ,draw opacity=1 ] (118.35,183) -- (122.86,194) -- (113.85,194) -- cycle ;
\draw  [color={rgb, 255:red, 14; green, 67; blue, 243 }  ,draw opacity=1 ] (141.4,183) -- (145.91,194) -- (136.89,194) -- cycle ;
\draw  [color={rgb, 255:red, 14; green, 67; blue, 243 }  ,draw opacity=1 ] (435.4,183) -- (439.91,194) -- (430.89,194) -- cycle ;
\draw  [color={rgb, 255:red, 14; green, 67; blue, 243 }  ,draw opacity=1 ] (180.49,183) -- (185,194) -- (175.98,194) -- cycle ;
\draw  [color={rgb, 255:red, 65; green, 117; blue, 5 }  ,draw opacity=1 ][pattern=_9g2mn5ju6,pattern size=3pt,pattern thickness=0.75pt,pattern radius=0pt, pattern color={rgb, 255:red, 65; green, 117; blue, 5}] (183.47,221) -- (188.12,231.42) -- (178.82,231.42) -- cycle ;
\draw  [color={rgb, 255:red, 14; green, 67; blue, 243 }  ,draw opacity=1 ] (243.6,221) -- (248.11,231.42) -- (239.09,231.42) -- cycle ;
\draw  [color={rgb, 255:red, 208; green, 2; blue, 27 }  ,draw opacity=1 ][fill={rgb, 255:red, 208; green, 2; blue, 27 }  ,fill opacity=1 ] (303.45,221) -- (307.95,231.42) -- (298.94,231.42) -- cycle ;
\draw  [color={rgb, 255:red, 139; green, 87; blue, 42 }  ,draw opacity=1 ] (348.49,227.23) .. controls (348.49,224.92) and (355.06,223.04) .. (363.16,223.04) .. controls (371.27,223.04) and (377.84,224.92) .. (377.84,227.23) .. controls (377.84,229.55) and (371.27,231.42) .. (363.16,231.42) .. controls (355.06,231.42) and (348.49,229.55) .. (348.49,227.23) -- cycle ;
\draw  [color={rgb, 255:red, 155; green, 155; blue, 155 }  ,draw opacity=1 ] (168.79,212.84) -- (407.55,212.84) -- (407.55,242.03) -- (168.79,242.03) -- cycle ;

\draw (153.06,25.55) node [anchor=north west][inner sep=0.75pt]    {$\textcolor[rgb]{0.96,0.65,0.14}{\langle s\rangle }$};
\draw (195.84,221) node [anchor=north west][inner sep=0.75pt]  [color={rgb, 255:red, 65; green, 117; blue, 5 }  ,opacity=1 ]  {$M$};
\draw (254.35,221) node [anchor=north west][inner sep=0.75pt]  [color={rgb, 255:red, 14; green, 67; blue, 243 }  ,opacity=1 ]  {$W$};
\draw (315.06,221) node [anchor=north west][inner sep=0.75pt]  [color={rgb, 255:red, 208; green, 2; blue, 27 }  ,opacity=1 ]  {$Z$};
\draw (384.33,221) node [anchor=north west][inner sep=0.75pt]    {$\textcolor[rgb]{0.55,0.34,0.16}{Q}$};

\end{tikzpicture}

\end{center}

In the figure above, the lines indicating the position of the triangles forming Z and M at each level go to the left as many times as to the right, which means that for every $r\in\N$ we have $\varphi_r(M\cup Z)=0$ (and so $M\cup Z$ is $(t,\eta)$-semibalanced).

\begin{proof}
    Let $F'=F \setminus (Q\cup Z)$ and fix $s\in \{-1,1\}^t$ such that $$\frac{\lambda(F \cap \downset{s})}{\lambda(\downset{s})}\geqslant 0.95.$$ 
    In particular, 
    $$\frac{\lambda(F' \cap \downset{s})}{\lambda(\downset{s})}\geqslant 0.9.$$
    Put $\mathcal{M}= \left\{M' \in \A_n : M' \subseteq F', \lambda(M')=\frac{k}{2^n}\right\}$. Since for $M' \in \mathcal{M}$ we have $M' \cap Z = \varnothing$. It follows that $\varphi_m(M'\cup Z)=\varphi_m(M')+\varphi_m(Z)$.

Define 
$$S(M')=\sum_{m=t+1}^n(\varphi_m(M'\cup Z))^2.$$

Choose a set $M \in \mathcal{M}$ such that 
$$S(M)=\min\{S(M'):M' \in \mathcal{M}\}.$$ 

We will show, that  $M\cup Z$ is $(t,\eta)$-semibalanced. Namely, we will show, that $S(M)<\frac{\eta^2}{n^2}$. This implies that for all $n\geqslant m>t$ we have 
$$|\varphi_m(M\cup Z)|\leqslant \sqrt{S(M)} < \frac{\eta}{n}\leqslant \frac{\eta}{m},$$ while for $m>n$ we have $|\varphi_m(M\cup Z)|=0$  (because $Z, M \in \A_n$).  Thus, $M\cup Z$ is $(t,\eta)$-semibalanced.\\
 We need to show that 
 \begin{equation}\label{S(M)}
   S(M)<\frac{k}{2^n}\frac{n}{2^{n-1}}. 
 \end{equation}
 Indeed, if the above inequality holds, then  $$S(M)< \frac{k}{2^n}\frac{n}{2^{n-1}}< \eta\frac{n}{2^{n-1}}= \frac{\eta}{n^2}\frac{n^3}{2^{n-1}}\leqslant \frac{\eta^2}{n^2}.$$

The inequality (\ref{S(M)}) may be written as:

\begin{eqnarray*}
S(M)&=&\sum_{m=t+1}^n \varphi_m(Z)\varphi_m(M\cup Z)+\sum_{m=t+1}^n \varphi_m(M)\varphi_m(M\cup Z)< \\ &<&\frac{k}{2^n}\frac{\sqrt{S(M)}}{4}+\frac{k}{2^n}\left(\frac{n}{2^{n-1}} - \frac{\sqrt{S(M)}}{4}\right).
\end{eqnarray*}

To prove it we will split it into two inequalities: 

\begin{equation}\label{b_m(a_m+b_m)}
    \sum_{m=t+1}^n \varphi_m(M)\varphi_m(M\cup Z) \leqslant \frac{k}{2^n}\left(\frac{n}{2^{n-1}} - \frac{\sqrt{S(M)}}{4}\right)
\end{equation}
and
\begin{equation}\label{a_m(a_m+b_m)}
\sum_{m=t+1}^n \varphi_m(Z)\varphi_m(M\cup Z)< \frac{k}{2^n}\frac{\sqrt{S(M)}}{4}.
\end{equation}

We will prove the inequalities (\ref{b_m(a_m+b_m)}) and (\ref{a_m(a_m+b_m)}) with the help of four claims. The first three are necessary to show the inequality (\ref{b_m(a_m+b_m)}) while the last one will prove inequality (\ref{a_m(a_m+b_m)}).

In Claim \ref{mala zmiana M} we make use of the minimality of $S(M)$ analyzing the situation when we change $M$ by one atom of $\A_m$.

\begin{claim}\label{mala zmiana M}
    For any $x, y \in \{-1, 1\}^n$ such that $\downset{x} \subseteq M$ and $\downset{y} \subseteq F' \setminus M$
   \begin{equation}\label{x(a+b)}
       \sum_{m=t+1}^n x_m \varphi_m(M\cup Z)\leqslant \frac{n}{2^{n-1}} + \sum_{m=t+1}^n y_m \varphi_m(M\cup Z),
   \end{equation}
where $x_m$ and $y_m$ are the $m$-th terms of the sequences $x$ and $y$ respectively.
   
\end{claim}
\begin{proof2}
   Let $M'=(M \setminus\downset{x})\cup\downset{y}$. We have

    \begin{eqnarray*}
		\varphi_m(M')&=& \int\limits_{M'} \delta_m d\lambda\ \ \ \ 
        =\int\limits_{M} \delta_m d\lambda-\int\limits_{\downset{x}} \delta_m d\lambda+\int\limits_{\downset{y}} \delta_m d\lambda\\
	&=& \varphi_m(M) +\frac{1}{2^n}(y_m - x_m).
    \end{eqnarray*}
    Since $M$ minimizes $S$, we have $S(M')-S(M)\geqslant 0$.
    Then
    \begin{eqnarray*}
        S(M')-S(M)&=&\sum_{m=t+1}^n\left ((\varphi_m(Z)+\varphi_m(M'))^2-(\varphi_m(Z)+\varphi_m(M))^2\right )\\
        &=& \sum_{m=t+1}^n \left (\varphi_m(M')^2-\varphi_m(M)^2+2\varphi_m(Z)(\varphi_m(M')-\varphi_m(M))\right )\\
        &=& \sum_{m=t+1}^n (\varphi_m(M')-\varphi_m(M))(\varphi_m(M')+\varphi_m(M)+2\varphi_m(Z))\\ 
        &=& \sum_{m=t+1}^n \frac{1}{2^n}(y_m-x_m)\left(\frac{1}{2^n}(y_m-x_m)+2\varphi_m(M)+2\varphi_m(Z)\right). 
    \end{eqnarray*}
    Multiplying the above by $2^{n-1}$ and using the fact that $(y_m-x_m)^2\in\{0, 4\}$  we get
     \begin{eqnarray*}
        0&\leqslant& \sum_{m=t+1}^n (y_m-x_m)\left(\frac{1}{2^{n+1}}(y_m-x_m)+\varphi_m(M)+\varphi_m(Z)\right)\\
        &\leqslant& \sum_{m=t+1}^n \frac{4}{2^{n+1}}+(y_m-x_m)(\varphi_m(M)+\varphi_m(Z))\\
        &\leqslant&  \frac{n}{2^{n-1}}+\sum_{m=t+1}^n (y_m-x_m)\varphi_m(M\cup Z).
    \end{eqnarray*}
\end{proof2}
Let $T\colon C \to C$ be a function that swaps the sign of the coordinates from $t+1$ to $n$ given by the formula:

$$
T(y)(i)=
    \begin{cases} -y(i) &\text{if} \ i\in(t, n],\\
        y(i) &\text{if} \  i\notin(t, n].
    \end{cases}
$$
Note that for $A\in \bor(C)$ we have $\varphi_m(A \cup T[A])=0$, where $m\in \{t+1, \dots, n\}$.

\begin{claim}
    $$
    \frac{\lambda\big(((F' \cap \downset{s})\cap T[F' \cap \downset{s}])\setminus (M \cup T[M])\big)}{\lambda(\downset{s})}\geqslant 0.75.
    $$
\end{claim}

\begin{proof2}
    Since $$\lambda(M)=\frac{k}{2^n}<\eta <\frac{0.05}{2^{t+1}},$$ we have $$\lambda(M \cup T[M])\leqslant \frac{0.05}{2^{t}}.$$ Since $$\frac{\lambda(F' \cap \downset{s})}{\lambda(\downset{s})}\geqslant 0.9,$$ we have $$\frac{\lambda((F' \cap \downset{s})\cap T[F' \cap \downset{s}])}{\lambda(\downset{s})}\geqslant 0.8.$$ Therefore
\begin{eqnarray*}
    \frac{\lambda(((F' \cap \downset{s})\cap T[F' \cap \downset{s}])\setminus (M \cup T[M]))}{\lambda(\downset{s})}
    &\geqslant& 0.8 -0.05=0.75.
\end{eqnarray*}
\end{proof2}

It is clear that if $y \in \downset{s}$ then $T(y) \in \downset{s}$. In the proof of the next claim we will use an obvious observation that if $\frac{\lambda(A \cap \downset{s})}{\lambda(\downset{s})}\geqslant 0.75$ and $\frac{\lambda(B \cap \downset{s})}{\lambda(\downset{s})}> 0.25$ then there exists $y \in A \cap B \cap \downset{s}$.
\begin{claim}
    There exists $y^M\in\{-1,1\}^n$ such that $\downset{y^M} \subseteq (F' \setminus M) \cap T[F' \setminus M]$ and

    $$
    \left|  \sum_{m=t+1}^n y_m^M\varphi_m(M\cup Z)\right| \geqslant \frac{\sqrt{S(M)}}{4}.
    $$
\end{claim}
\begin{proof2}
  Recall that $S(M)=\sum_{m=t+1}^n \varphi_m(M\cup Z)^2$. By Lemma \ref{wcale_nie_Paley_Zygmund} for $\xi=1/16$ and
    $$
    d_m=
    \begin{cases} 
    \varphi_m(M\cup Z) &\text{if} \ m\in(t, n],\\
    0          &\text{if} \ m\in[0, t],
    \end{cases}
    $$
    we have
    $$
    \lambda_n\left(\left\{y\in\{-1, 1\}^{n}: \left|\sum_{m=t+1}^n y_m \varphi_m(M\cup Z)\right|^2\geqslant \frac{1}{16} S(M)\right\}\right)\geqslant\frac{1}{3} \left(\frac{15}{16}\right)^2.
    $$
    Hence 
    $$
    \lambda_n\left(\left\{y\in\{-1, 1\}^{n}: \left|\sum_{m=t+1}^n y_m \varphi_m(M\cup Z)\right|\geqslant \frac{\sqrt{S(M)}}{4} \right\}\right)>0.25.
    $$
    Now Claim 2 implies the existence of the desired $y^M$.
\end{proof2}
Note that $(F' \setminus M)\cap T[F' \setminus M] =T[(F' \setminus M)\cap T[F' \setminus M]]$.
So since $\downset{y^M} \subseteq (F' \setminus M) \cap T[F' \setminus M]$ we also have $T[\downset{y^M}] \subseteq (F' \setminus M) \cap T[F' \setminus M]$. Moreover $\sum_{m=t+1}^n y^M_m \varphi_m(M\cup Z)=-\sum_{m=t+1}^n T(y^M_m) \varphi_m(M\cup Z)$. Thus, by replacing $y^M$ with $T(y^M)$ if needed Claim 3 implies that
\begin{equation}\label{suma iloczynu (a+b) z y}
    \sum_{m=t+1}^n y^M_m \varphi_m(M\cup Z)\leqslant\ -\frac{\sqrt{S(M)}}{4}.
\end{equation}
From inequalities (\ref{x(a+b)}) and (\ref{suma iloczynu (a+b) z y}) for every $x\in \{-1,1\}^n$ with $\downset{x}\subseteq M$, we get
\begin{equation}\label{suma iloczynu (a+b) z x}
    \sum_{m=t+1}^n x_m \varphi_m(M\cup Z)\leqslant \frac{n}{2^{n-1}} + \sum_{m=t+1}^n y^M_m \varphi_m(M\cup Z)\leqslant \frac{n}{2^{n-1}} - \frac{\sqrt{S(M)}}{4}.
\end{equation}
Let $\{x^{(i)}: i\in\{1, \dots k \}\}$ be an enumeration of all $x\in\{-1,1\}^n$ such that $\downset{x}\subseteq M$. 

Taking into account that $M\in \mathcal{M}\subseteq\A_n$, we have
$$
M= \bigcup_{\substack{x\in \{-1,1\}^n\\ \downset{x}\subseteq M}} \downset{x}= \bigcup_{i=1}^{k} \downset{x^{(i)}}
$$
Since for all $x\in \{-1, 1\}^n$ and $m\leqslant n$ we have $x_m= 2^n \int_{\downset{x}} \delta_m d \lambda$ we conclude that
\begin{equation}\label{suma z x}
\sum_{i=1}^{k} x_m^{(i)}=2^n \int_{M} \delta_m d \lambda=2^n \varphi_m(M).
\end{equation}
By (\ref{suma z x}) and (\ref{suma iloczynu (a+b) z x}) we obtain

\begin{equation} \label{2^n b(a+b)}
\begin{split}
\sum_{m=t+1}^n 2^n \varphi_m(M)\varphi_m(M\cup Z)&=\sum_{m=t+1}^n\sum_{i=1}^{k} x_m^{(i)}\varphi_m(M\cup Z)=\\
    =\sum_{i=1}^{k} \sum_{m=t+1}^n x_m^{(i)}\varphi_m(M\cup Z) &\leqslant k\left(\frac{n}{2^{n-1}} - \frac{\sqrt{S(M)}}{4}\right).
\end{split}
\end{equation}
Multiplying both sides of the inequality (\ref{2^n b(a+b)}) by $2^{-n}$ we get

\begin{equation*}\label{b(a+b)}
\sum_{m=t+1}^n \varphi_m(M)\varphi_m(M\cup Z)\leqslant \frac{k}{2^n}\left(\frac{n}{2^{n-1}} - \frac{\sqrt{S(M)}}{4}\right),
\end{equation*}
that is the inequality (\ref{b_m(a_m+b_m)}).

Now we will prove (\ref{a_m(a_m+b_m)}).
\begin{claim}
 $$\sum_{m=t+1}^n \varphi_m(Z)\varphi_m(M\cup Z) < \frac{k}{2^n}\frac{\sqrt{S(M)}}{4}.$$   
\end{claim}
\begin{proof2}
    By the Cauchy-Schwarz inequality we have
 \begin{eqnarray*} 
    \sum_{m=t+1}^n \varphi_m(Z)\varphi_m(M\cup Z) &\leqslant& \sqrt{\sum_{m=t+1}^n (\varphi_m(Z))^2}\sqrt{\sum_{m=t+1}^n \varphi_m(M\cup Z)^2}=\\
    &=&\sqrt{\sum_{m=t+1}^n (\varphi_m(Z))^2}\sqrt{S(M)}.
\end{eqnarray*}

By Lemma \ref{orthonormal} the sequence $(\delta_n)_{n\in\N}$ is orthonormal in the Hilbert space $\mathcal{L}_2(C)$, so by the Bessel inequality we get
\begin{eqnarray*} 
    \sum_{m=t+1}^n (\varphi_m(Z))^2&=&\sum_{m=t+1}^n \left(\int_{Z} \delta_m d \lambda\right)^2=\sum_{m=t+1}^n \left(\int_{C} \chi_{Z} \delta_m d \lambda\right)^2=\\
    &=&\sum_{m=t+1}^n (\downset{\chi_{Z}, \delta_m})^2\leqslant\|\chi_{Z}\|_{2}^2=\lambda({Z}).
\end{eqnarray*}

Since $\lambda(Z)<\eta^2/64$, we have $\sqrt{\lambda(Z)}<\eta/8<k/2^{n+2}$ and so
 \begin{equation*}\label{a(a+b)}
 \sum_{m=t+1}^n \varphi_m(Z)\varphi_m(M\cup Z)\leqslant\sqrt{\lambda(Z)}\sqrt{S(M)}< \frac{k}{2^n}\frac{\sqrt{S(M)}}{4}.
 \end{equation*}
\end{proof2}

Adding inequalities (\ref{b_m(a_m+b_m)}) and (\ref{a_m(a_m+b_m)}) side by side we get the following estimate

\begin{eqnarray*}
\sum_{m=t+1}^n \varphi_m(M\cup Z)^2
&<&\frac{k}{2^n}\frac{\sqrt{S(M)}}{4} +\frac{k}{2^n}\left(\frac{n}{2^{n-1}} - \frac{\sqrt{S(M)}}{4}\right) =\frac{k}{2^n}\frac{n}{2^{n-1}},
\end{eqnarray*}
which shows (\ref{S(M)}) and finishes the proof. 
\end{proof}

\begin{lemma}\label{tristane}
Let $\B^*\subseteq \B$ be balanced Boolean subalgebras of $\bor(C)$ containing $\clop(C)$ and let $\F$ be a finite subalgebra of $\B$. Let $P\in \F\cap \B^*$. Let $t\in \N, \delta>0$. Suppose that for every $A\in \F$ such that $\lambda(A)>0$ there is $s_A\in \{-1,1\}^t$ such that 

\begin{equation*}
\begin{gathered}
\frac{\lambda(A\cap \downset{s_A})}{\lambda(\downset{s_A})}\geqslant 0.99.
\end{gathered}
\end{equation*}
 Then there is $\theta>0$ such that for any $L, Q\in \B^*$, if $\max\{\lambda (L), \lambda(Q)\}<\theta$ and $L\cap P=\varnothing$, then there is $M\in \B^*$ such that
 \begin{enumerate}
     \item $M\cap (P\cup Q)=\varnothing$,
     \item $\lambda(M)<\delta$,
     \item $\forall F\in\F \ (M\cup L)\cap F$ is $(t,\delta)$-semibalanced,
     \item $\forall F\in\F \ F\backslash (M\cup L)$ is $(t,\delta)$-semibalanced.
 \end{enumerate}
\end{lemma}

\begin{proof}
Let $ \eta<\min\{\delta/(4|\F|), 1/2^{t+10}\}, \theta<\eta^2/64$ and let $n_0$ be large enough so that for every $n>n_0$ 
\begin{enumerate}\setcounter{enumi}{4}
  \item there is $k\in \N$ such that $\frac{\eta}{2} <\frac{k}{2^n}< \eta$,
  \item  $\frac{n^3}{2^{n-1}}\leqslant \eta$,
  \item  $\frac{\eta}{n}<\frac{\eta^2}{64}-\theta.$
\end{enumerate}
 Fix $Q$ and $L$ satisfying the hypothesis of the lemma. Denote by $\HH_0$ the Boolean algebra generated by $\F\cup \{Q, L\}$. Since $\B$ is balanced, there is $n> n_0$ such that $\HH_0$ is $(n, \eta)$-balanced. Let $\HH$ be the Boolean algebra generated by $\HH_0 \cup \A_n$. By Lemma \ref{miecze_i_pierogi} 
 $$\HH \ \text{is} \ (n, \eta)\text{-balanced}.$$
Let $h_n$ be defined as in Lemma \ref{homomorphism}. By the same lemma, for $F\in \F$ with $\lambda(F)>0$ we have $$\frac{\lambda(h_n(F)\cap \downset{s_F})}{\lambda(\downset{s_F})} \geqslant \frac{\lambda(F\cap \downset{s_F})}{\lambda(\downset{s_F})}-\frac{\lambda(F\triangle h_n(F))}{\lambda(\downset{s_F})}\geqslant 0.99-\frac{2^t\eta}{n}\geqslant 0.95.$$
By Lemma \ref{homomorphism} and (7)
$$\lambda(h_n(L)\cap h_n(F))\leqslant \lambda(h_n(L)) \leqslant \lambda(L)+\frac{\eta}{n}<\theta+\frac{\eta^2}{64}-\theta= \frac{\eta^2}{64}.$$
By Lemma \ref{lemme} (applied to $Z=h_n(L)\cap h_n(E))$ for every $E\in \at(\F)$ there is $M_E\in \A_n$ such that 
\begin{enumerate}
    \item [(a)] $M_E\subseteq h_n(E)\backslash h_n(Q)$,
    \item [(b)] $\lambda(M_E)<\eta<\delta/|\F|$,
    \item [(c)] $M_E\cup( h_n(L)\cap h_n(E))$ is $(t,\delta/(4|\F|))$-semibalanced.
\end{enumerate} 
Put $$M_0=\bigcup_{E\in \at(\F) } M_E\backslash Q, M=M_0\backslash P.$$
Then $M\cap (P\cup Q)= \varnothing$ and $\lambda(M)<\delta$. To show (3) and (4) fix $F\in \F$ and $r>t$. 
Consider 2 cases.
\medskip\\
\textbf{Case 1.} $r\leqslant n$.
\medskip\\
For every $E\in \at(\F)$ we have $M_E\in \A_n$, so $h_n(M_E)=M_E$. From (a) we get 
\begin{eqnarray*}
    h_n(M_0\cap h_n(F)) &=& h_n(M_0\cap F) = h_n\left( \bigcup_{E\in \at(\F)} (M_{E}\backslash Q)\cap F\right) =\\ 
    &=& \bigcup_{\substack{E\in \at(\F)\\ E\subseteq F}} (h_n(M_{E})\cap h_n(F))\backslash h_n(Q) = \\ &=& \bigcup_{\substack{E\in \at(\F)\\ E\subseteq F}} ((M_{E}\cap h_n(F))\backslash h_n(Q) =\\ &=& \bigcup_{\substack{E\in \at(\F)\\ E\subseteq F}} (M_E\cap h_n(F))\backslash h_n(Q) =\bigcup_{\substack{E\in \at(\F)\\ E\subseteq F}} M_E.
\end{eqnarray*} 
By Lemma \ref{homomorphism} for any $A\in \HH$ we have 
$$|\varphi_r(A)|\leqslant |\varphi_r(h_n(A))|+\lambda(A\triangle h_n(A)) < |\varphi_r(h_n(A))|+ \frac{\eta}{n}. $$
Putting $A= (M_0\cup L)\cap F$ and using (c) we get
\begin{eqnarray*}
    |\varphi_r((M_0\cup L)\cap F)|
    &\leqslant&  |\varphi_r (h_n((M_0\cup L)\cap F))| + \frac{\eta}{n}<\\
    &<& |\varphi_r ((h_n(M_0)\cup h_n(L))\cap h_n(F))|+ \frac{\delta}{4n} = \\ 
    &=& \Bigg|\varphi_r\Bigg(\bigcup_{\substack{E\in \at(\F)\\ E\subseteq F}} M_E\cup (h_n(L)\cap h_n(E))\Bigg)\Bigg|+\frac{\delta}{4n} \leqslant \\
    &\leqslant& \sum_{\substack{E\in \at(\F)\\ E\subseteq F}} |\varphi_r(M_E\cup (h_n(L)\cap h_n(E)))|+\frac{\delta}{4n} \leqslant \\
    &\leqslant& |\F|\frac{\delta}{4r|\F|}+\frac{\delta}{4n} \leqslant \frac{\delta}{2r}.
\end{eqnarray*}
\textbf{Case 2.} $r>n$.
\medskip\\
In this case, since $\HH$ is $(n, \delta/2)$-balanced, we have 
$$|\varphi_r ((M_0\cup L)\cap F)|\leqslant \sum_{s\in \{-1,1\}^n} |\varphi_r ((M_0\cup L)\cap F\cap\downset{s})|< 2^n\lambda(\downset{s})\frac{\delta}{2r}=\frac{\delta}{2r}.$$
Hence 
\begin{gather}\label{panie_lasu}
\tag{8}
(M_0\cup L)\cap F \ \text{is} \ \left(t,\frac{\delta}{2}\right)\text{-semibalanced for} \ F\in \F.
\end{gather}
Since $L\cap P=\varnothing$ and $M=M_0\backslash P$, we get $$(M\cup L)\cap F= (M_0\cup L)\cap (F\backslash P).$$ 
Since $P\in \F$ we get that if $F\in \F$, then $F\backslash P\in \F$, so by (\ref{panie_lasu}) applied to $F\backslash P$ we get that
\begin{gather}\label{po_omacku}
\tag{9}
(M\cup L)\cap F \ \text{is} \ \left(t,\frac{\delta}{2}\right)\text{-semibalanced for} \ F\in \F,
\end{gather}
which implies (3).

For (4) note that $C\backslash (F\backslash (M\cup L)) = (C\backslash F)\cup ((M\cup L)\cap F)$. Since $C \backslash F$ and $(M\cup L)\cap F$ are disjoint we have $$|\varphi_r(C\backslash (F\backslash (M\cup L)))|\leqslant |\varphi_r(C\backslash F)| + |\varphi_r((M\cup L)\cap F)|.$$
Since $\HH$ is $(t,\delta/2)$-balanced and $C\backslash F\in \F \subseteq \HH$ we have $|\varphi_r(C\backslash F)|<\delta/(2r)$ and from (\ref{po_omacku}) we get that $|\varphi_r((M\cup L)\cap F)|< \delta/(2r)$. Hence 
$$|\varphi_r(C\backslash (F\backslash (M\cup L)))|< \frac{\delta}{r},$$
so $C\backslash (F\backslash (M\cup L))$ is $(t,\delta)$-semibalanced. By Lemma \ref{krwawy_baron} the set $F\backslash (M\cup L)$ is $(t,\delta)$-semibalanced, which shows (4) and finishes the proof.
\end{proof}

\begin{proof3}
 Denote by $\F$ the subalgebra of $\B$ generated by $$\{G, P\}\cup \bigcup_{n\leqslant k} \B_n\cup  \A_{m_k}.$$ Put 

\begin{gather*}\label{egazmin_z_alchemii}
\tag{4}
      \varepsilon= \min\left\{\frac{1}{100}\inf \{\lambda(A): A\in \F, \lambda(A)>0 \}, 2^{-m_k-k-1}\right\}.
\end{gather*}

Since $\B$ is balanced, there is $t\in \N, t>m_k$ such that 
\begin{gather}\label{niech_zyje_sztuka}
\tag{5}
  \F \ \text{is} \ (t, \varepsilon)\text{-balanced}.
\end{gather}

By Lemma \ref{algebra_E} for every $A\in \F$ there is $s_A\in \{-1,1\}^t$ such that

\begin{gather*}
\frac{\lambda(A\cap \downset{s_A})}{\lambda(\downset{s_A})}\geqslant 0.99.
\end{gather*}

By Lemma \ref{small_perturbation} and the assumption (B) of the proposition there is $\mem>0$ such that

\begin{equation}\tag{6}\label{aehrsdfg}
\begin{gathered}
\forall n\leqslant k \ \forall A\in \FF(\B_n, G) \ \forall B\in \bor(C) \\ 
\lambda(B)<\mem \implies A\cup B, A\backslash B \ \text{are} \ (m_n, t, 2^{-n})\text{-balanced}.
\end{gathered}
\end{equation}
By Lemma \ref{tristane} (applied to $\delta=\min\left\{\eta, \mem/2, 2^{-m_{k}-k-1}\right\}$), there is $0<\theta<\mem/2$ such that for any $L,Q\in \B^*$, if $\max\{\lambda (L), \lambda(Q)\}<\theta$ and $L\cap P=\varnothing$, then there is $M\in \B^*$ satisfying

\begin{itemize}
    \item [(1')] $M\cap (P\cup Q)=\varnothing$,
    \item [(2')] $\lambda(M)< \min\left\{\eta, \mem/2\right\}$,
    \item [(3')] $\forall F\in\F \ (M\cup L)\cap F$ is $(t,2^{-m_{k}-k-1})$-semibalanced,
    \item [(4')] $\forall F\in\F \ F\backslash(M\cup L)$ is $(t,2^{-m_{k}-k-1})$-semibalanced.
\end{itemize}
In particular, the conditions (1) and (2) of the proposition are satisfied.

In order to show that such $M$ satisfies (3), fix $L,Q\in \B^*$ such that $L\cap P=\varnothing$ and $\max\{\lambda (L), \lambda(Q)\}<\theta$. 
Then we have $$\lambda(L\cup M)\leqslant \lambda(L) + \lambda(M)< \mem/2+\mem/2= \mem,$$
so by (\ref{aehrsdfg}) for every $n\leqslant k$ the family
\begin{gather}\label{nabial_i_ciemne_sily}
\tag{7}
    \FF(\B_n, G\cup L \cup M) \ \text{is} \ (m_n, t, 2^{-n})\text{-balanced}.
\end{gather}

Since $\A_{m_n}, \B_n\subseteq \F$ for $n\leqslant k$ and $G\in \F$, for every $A\in\B_n$ and $s\in \{-1,1\}^{m_n}$ we have $\downset{s}\cap A\cap G, A\backslash G\in \F$.
Hence by (\ref{egazmin_z_alchemii}), (\ref{niech_zyje_sztuka}) and Lemma \ref{balanced+semibalanced}
\begin{gather}\label{incydent_w_bialym_sadzie}
\tag{8}
    \downset{s}\cap A\cap G, \downset{s}\cap A\backslash G \ \text{are} \ (t,2^{-m_n-n-1})\text{-semibalanced.} 
\end{gather}
Since $\downset{s}\cap A\in \F$, by (3') and (4')
\begin{gather}\label{sprawy_rodzinne}
\tag{9}
    \downset{s}\cap A\cap (L\cup M), \downset{s}\cap A\backslash (L\cup M) \ \text{are} \ (t,2^{-m_n-n-1})\text{-semibalanced.}
\end{gather}
Since $L\cap P= M\cap P=\varnothing$ and $G\subseteq P$, the sets $L\cup M$ and $G$ are disjoint, so by (\ref{incydent_w_bialym_sadzie}) and (\ref{sprawy_rodzinne}) for $A\in \B_n, s\in \{-1,1\}^{m_n}$ 
\begin{gather*}
    \downset{s}\cap A\cap (G\cup L \cup M), \downset{s}\cap A\backslash (G\cup L \cup M) \ \text{are} \ (t,2^{-m_n-n})\text{-semibalanced}. 
\end{gather*}
By the above, (\ref{nabial_i_ciemne_sily}) and Lemma \ref{balanced+semibalanced} 
$$ \FF(\B_n, G\cup L \cup M) \ \text{is} \ (m_n, 2^{-n})\text{-balanced},$$
which shows (3) and completes the proof.
\end{proof3}

\section{Extensions of countable balanced Boolean algebras}\label{construction}

In this section, we will show how to enlarge a given countable balanced Boolean algebra $\B$ to a balanced Boolean algebra $\B^*$, so that $(\B, \B^*, \nu)$ satisfies the property $(\mathcal{G}^*)$, where $\nu$ is a normal sequence of measures on $\B$. We will also show how to deal with finitely many measures simultaneously, which will be important in the forcing construction.

\begin{lemma}\label{zbior_X}\emph{(Folklore
)} Suppose $K$ is a compact Hausdorff space with an open basis $\BB$ that is closed under finite unions. Let $\widetilde{\nu}\in M(K)$ and let $\MM\subseteq M(K)$ be a finite set of measures such that $\widetilde{\nu} \bot \widetilde{\mu}$ for every $\widetilde{\mu}\in \MM$. Then for every $\varepsilon>0$ there is $X\in\BB$ such that $|\widetilde{\nu}|(X)<\varepsilon$ and for every $\widetilde{\mu}\in\MM$ we have $|\widetilde{\mu}|(K\backslash X)<\varepsilon$. 
\end{lemma}
 
\begin{proof}
    First, we will show that the lemma holds when $\MM=\{\widetilde{\mu}\}$. 
    Since $\widetilde{\nu} \bot \widetilde{\mu}$ we have  $|\widetilde{\nu}| \bot |\widetilde{\mu}|$, so there exists a Borel support $A$ of $|\widetilde{\mu}|$, such that $K \setminus A$ is a Borel support of $|\widetilde{\nu}|$. By the regularity there is closed $B \subseteq A$ such that $|\widetilde{\mu}|(B) > |\widetilde{\mu}|(K)-\varepsilon$. We can find a closed set $D \subseteq K \setminus B$ such that $|\widetilde{\nu}|(D) > |\widetilde{\nu}|(K) - \varepsilon$. Note that  $|\widetilde{\mu}|(D)\leqslant \varepsilon$. Since $U=K\setminus D$ is an open superset of $B$, there are finitely many sets $B_1, \dots B_k\in \BB, B_i\subseteq U$ for $i\leqslant k$, which cover $B$. Let $X=\bigcup_{i \leqslant k} B_i$.  Note that $|\widetilde{\mu}|(X) \geqslant |\widetilde{\mu}|(K)-\varepsilon$, and $X \subseteq K \setminus D$, so $|\widetilde{\nu}|(X)<\varepsilon$ and $|\widetilde{\mu}|(K\backslash X)<\varepsilon$.

   When $\MM=\{\widetilde{\mu}_1, \dots \widetilde{\mu}_n\}$, by the first part of the proof for each pair $(\widetilde{\nu}, \widetilde{\mu}_i)$ where $i \leqslant n$ we can find $X_i \in \BB$ such that $|\widetilde{\nu}|(X_i)< \frac{\varepsilon}{n}$ and $|\widetilde{\mu}_i|(K\backslash X_i)<\frac{\varepsilon}{n}$. Then for $X= \bigcup_{i \in {1, \dots n}} X_i$ we have $|\widetilde{\nu}|(X)< \varepsilon$ and $|\widetilde{\mu}_i|(K\backslash X)<\frac{\varepsilon}{n}<\varepsilon$. 
\end{proof}

\begin{lemma}\emph{(Pe\l czy\'nski)}\cite[Lemma 5.3]{Pelczynski_lemma}
    Let $(\widetilde{\nu}_n)_{n\in\N}$ be a bounded sequence of Radon measures on a compact space $K$. Suppose there are pairwise disjoint Borel sets $(E_n)_{n\in\N}$ and $c>0$ such that $\widetilde{\nu}_n(E_n)\geqslant c$ for every $n\in\N$. Then for every $\delta>0$ there is a subsequence $(\widetilde{\nu}_{n_k})_{n\in\N}$ and a sequence of pairwise disjoint open sets $(U_k)_{k\in\N}$ such that $\widetilde{\nu}_{n_k}(U_k)\geqslant c-\delta$ for every $k\in\N$. 
\end{lemma}

We will use the following application of the above lemma, in which $K$ is the Stone space of a Boolean algebra. 

\begin{corollary}\label{Pelczynski}
    Let $(\nu_n)_{n \in \N}$ be a sequence of measures on $\B$  and $(E_n)_{n\in\N}$ be a sequence of disjoint Borel sets in $\st(\B)$. Let $P\in \B$ be such that $E_n\cap [P]=\varnothing$ for every $n\in\N$. Let $c,\delta>0$. If $|\widetilde{\nu}_n| (E_n)\geqslant c$ for every $n \in \N$, then there exist a subsequence $(\nu_{n_k})_{k \in \N}$ and a sequence $(V_k)_{k \in \N}$ of pairwise disjoint elements of $\B$ such that $V_k\cap P=\varnothing$ and $|\nu_{n_k}|(V_k)\geqslant c- \delta$ for all $k \in \N$. 
\end{corollary}


The next lemma will let us build an antichain needed to satisfy the property $(\mathcal{G})$.

\begin{lemma}\label{2_elements_from_antichain}
   Let $\nu=(\nu_n)_{n\in\N}$ be a normal sequence of measures on a Boolean algebra $\B\subseteq\bor(C)$.
    Let $\MM$ be a finite set of positive measures on $\B$ and assume that $(|{\nu}_n|)_{n\in\N}$ has a subsequence pointwise convergent to a measure $\OO\in \MM$. Let $d\in \N$ and $\varepsilon>0$. Let $P\in \B$  be such that 
 $$\OO(P)< 0.1.$$
 Then there are $H_0,H_1\in \B$ and $a,b>d$ such that 
 \begin{enumerate}
     \item $H_0,H_1,P$ are pairwise disjoint,
     \item $\forall \mu\in\MM \ \mu(H_0\cup H_1)<\varepsilon$, 
     \item $\lambda (H_0\cup H_1)<\varepsilon$,
    \item $|\nu_a|(H_0), |\nu_b|(H_1)\geqslant 0.9$.
 \end{enumerate}
 \end{lemma}
       \begin{proof}
       We may assume that $\lambda \upharpoonright \B \in \MM$, so it is enough to show (1), (2) and (4). 

       Since ${\nu}_\infty$ is the pointwise limit of a sequence of probability measures, we have $\OO(C)=1$, and so $\OO(C\backslash P)>0.9$. Hence there is $\delta>0$ such that for infinitely many $n\in \N$ we have
       $$|\nu_n|(C\backslash P)>0.9+\delta.$$
      
Since $(\widetilde{\nu}_n)_{n\in\N}$ has pairwise disjoint Borel supports, there are pairwise disjoint sets $(E_n)_{n\in\N}\subseteq \bor(C)$ such that $$|\widetilde{\nu}_n|(E_n\backslash [P])>0.9+\delta$$
for infinitely many $n\in\N$.
    By Corollary \ref{Pelczynski} there is an antichain $(V_l)_{l\in\N}\subseteq \B$ and a subsequence $(\nu_{n_l})_{l\in\N}$ such that $V_l\cap P=\varnothing$ and $|\nu_{n_l}|(V_l)\geqslant 0.9$ for every $l\in\N$. Since $(V_l)_{l\in\N}$ is an antichain, we have for every $\mu\in\MM$ $$\lim_{l\to \infty} \mu(V_l)=0.$$  
 In particular, if $l_1,l_2$ are big enough, then we have $\mu(V_{l_1}\cup V_{l_2})<\varepsilon$ for every $\mu\in\MM$, so it is enough to put $a=n_{l_1}, b=n_{l_2}, H_0=V_{l_1}, H_1=V_{l_2}$ where $l_1 \neq l_2$ are so big that $\min\left\{n_{l_1},n_{l_2}\right\}> \max\{d,n\}$. 

 \end{proof}

In the next two lemmas we describe how to pick a sequence of sets, whose union will be a witness for the property $(\mathcal{G}^*)$ for a given sequence of measures. 

For the purpose of the construction under {\sf CH}, in the following lemma it is enough to take $\MM=\{\OO\}$. The case when $\MM$ consists of more than one measure will be used in Section \ref{forcing}.

\begin{lemma}\label{construction-finite-step}
Suppose we are given:

\begin{enumerate}[leftmargin=21pt,label=(\Alph*)]
\item natural numbers $k, d\in\N,$
    \item subalgebras $\B^*, \B\subseteq\bor(C)$ and finite subalgebras $\B_n\subseteq\B$ for $n\leqslant k+1$ such that
    \begin{itemize}[leftmargin=13.5pt]
        \item $\B$ is balanced,
        \item $\clop(C)\subseteq \B^*\subseteq \B$,
    \end{itemize}
    \item a normal sequence of measures $(\nu_n)_{n\in\N}$ on $\B^*$ and a finite set $\MM$ of probability measures on $\B^*$ such that $(|{\nu}_n|)_{n\in\N}$ has a subsequence pointwise convergent to a measure $\OO\in \MM$,
\item a strictly increasing sequence of natural numbers $(m_n)_{n\leqslant k}$ and sets $\widehat{G},\widehat{H}\in \B^*$ such that 
\begin{itemize}[leftmargin=13.5pt]
    \item  $\forall n\leqslant k \ \FF(\B_n, \widehat{G}) \ \text{is} \ (m_n, 2^{-n})\text{-balanced},$
    \item  $\forall \mu\in \MM \ \mu(\widehat{G} \cup \widehat{H})<0.1.$
\end{itemize}
\end{enumerate}
    Then there are $a,b>d; m_{k+1} \in \N; G',H_0, H_1\in \B^*$ such that:
    \begin{enumerate}
        \item $m_{k+1}>m_k,$
        \item $\forall n\leqslant k+1 \ \FF(\B_n, \widehat{G}\cup G') \ \text{is} \ (m_n, 2^{-n})\text{-balanced},$
        \item $\forall \mu\in \MM \ \mu(\widehat{G}\cup G' \cup \widehat{H}\cup H_0\cup H_1)<0.1,$
        \item $\widehat{H}, H_0, H_1 \ \text{are pairwise disjoint},$
        \item $G'\cap (\widehat{G}\cup \widehat{H} \cup H_1) = \varnothing,$
        \item $\widehat{G}\cap (H_0 \cup H_1) = \varnothing,$
        \item $|\nu_a|(H_0), |\nu_b|(H_1) \geqslant 0.9,$
        \item $|\nu_a(G'\cap H_0)|\geqslant 0.3.$
    \end{enumerate}
\end{lemma}
    
\begin{proof}   
Let $\E$ be the subalgebra of $\B$ generated by $\B_{k+1}\cup \{\widehat{G}, \widehat{H}\}$. Since $\B$ is balanced, by Remark \ref{dla_dobra_nauki} there is $m_{k+1}\in \N, m_{k+1}>m_k$ such that 
\begin{equation*}
\begin{gathered}
  \E \ \text{is} \ \left(m_{k+1}, \frac{1}{2^{k+1}}\right)\text{-balanced}.  
\end{gathered}
\end{equation*}
In particular, 
\begin{equation}\label{a1}
\tag{9}
\begin{gathered}
  \FF(\B_{k+1}, \widehat{G}) \ \text{is} \ \left(m_{k+1}, \frac{1}{2^{k+1}}\right)\text{-balanced}.  
\end{gathered}
\end{equation}

Put 
\begin{equation}\label{a6}
\tag{10}
\begin{gathered}
\xi=0.1-\max_{\mu\in \MM} \mu(\widehat{G}\cup \widehat{H}). 
\end{gathered}
\end{equation}
For every $\mu\in \MM$ consider its Lebesgue decomposition (cf. \cite[Theorem 6.10]{Rudin_Real_Complex_Analysis})
\begin{equation}\label{a7}
\tag{11}
\begin{gathered}
\mu=\mu_1+\mu_2, \ \text{where} \ \mu_1 \ll \lambda \ \text{and} \ \mu_2 \bot \lambda. 
\end{gathered}
\end{equation}
In particular (cf. \cite[Theorem 6.11]{Rudin_Real_Complex_Analysis}) there is $\eta>0$ such that 
\begin{equation}\label{a8}
\tag{12}
\begin{gathered}
\forall \mu\in \MM \ \forall A\in\B^* \ \lambda(A)<\eta\implies \mu_1(A)<\xi/4.
\end{gathered}
\end{equation}

By (\ref{a1}), the first part of (D) and  Proposition \ref{poprawka} (applied to $P=\widehat{G}\cup \widehat{H}$, $G=\widehat{G}$) there is $\theta>0$ such that whenever $L,Q\in \B^*$ and $\max\{\lambda(L),\lambda(Q)\}<\theta$, there is $M_{(L,Q)}\in \B^*$ such that:

\begin{enumerate}[leftmargin=23pt]
\setcounter{enumi}{12}
     \item  $M_{(L,Q)}\cap (\widehat{G}\cup \widehat{H}\cup Q)=\varnothing$,
     \item \label{x14} $\lambda\left(M_{(L,Q)}\right)<\eta$,
     \item \label{x15} $\forall n\leqslant k+1 \ \FF\left(\B_n, \widehat{G}\cup L \cup M_{(L,Q)}\right)$ is $(m_n, 2^{-n})$-balanced.  
 \end{enumerate}

By (\ref{a7}) and Lemma \ref{zbior_X} there is $X\in \B^*$ such that for every $\mu\in \MM$
\begin{equation}\label{a9}
\tag{16}
\begin{gathered}
\lambda(X)<\theta, \mu_2(C\backslash X)<\xi/4.
\end{gathered}
\end{equation}

By Lemma \ref{2_elements_from_antichain} (applied to $\varepsilon=\min\left\{\theta-\lambda(X), \xi/4\right\}$, $P=\widehat{G}\cup \widehat{H}$) there are $a,b>d$ and $H_0, H_1\in \B^*$ such that

\begin{itemize}
    \item $\lambda (H_0\cup H_1\cup X)< \theta$,
    \item $\forall \mu\in\MM \ \mu(H_0\cup H_1)<\xi/4$,
    \item $H_0\cap H_1 = (H_0\cup H_1)\cap (\widehat{G}\cup \widehat{H})=\varnothing$,
    \item $|\nu_a|(H_0), |\nu_b|(H_1) \geqslant 0.9$.
\end{itemize}

Let $L\in \B^*$ be such that 
\begin{equation}\label{a12}
\tag{17}
\begin{gathered}
L\subseteq H_0 \ \text{and} \ |\nu_a(L)|\geqslant 0.3. 
\end{gathered}
\end{equation}

Let $M=M_{(L,H_0\cup H_1\cup X)}$. Then, in particular,
    \begin{gather}\tag{18}\label{a18}
        M\cap (X\cup \widehat{H}\cup \widehat{G}\cup H_0\cup H_1)=\varnothing.
    \end{gather}

We put $G'=L\cup M$. 

We need to verify that these definitions satisfy conditions $(1)$-$(8)$.

$(1)$ follows directly from the choice of $m_{k+1}$.
$(2)$ follows from (\ref{x15}).       

For $(3)$ fix any $\mu\in \MM$. By (\ref{a8}) and (\ref{x14}) we have 
$$\mu_1(M)<\xi/4.$$
By (\ref{a9}) and (\ref{a18}) we have
$$\mu_2(M)<\xi/4.$$
By (\ref{a12})
$$\mu(L) < \xi/4.$$
Hence $$\mu(G')= \mu(M)+\mu(L)=\mu_1(M)+\mu_2(M)+\mu(L)<3\xi/4.$$
Finally, from (\ref{a6}) we get 
$$\mu(\widehat{G}\cup G'\cup \widehat{H}\cup H_0\cup H_1) = \mu(\widehat{G}\cup \widehat{H}) + \mu(G')+ \mu(H_0\cup H_1)< 0.1-\xi+3\xi/4 + \xi/4=0.1.$$

Conditions $(4)$-$(7)$ follow directly from the choice of $a,b, H_0, H_1, G'$.

(8) follows from (\ref{a12}) and the fact that $G'\cap H_0=L$.
\end{proof}

\begin{lemma}\label{one_step_extension}
    Let $\B^*\subseteq \B\subseteq \bor(C)$ be balanced countable Boolean algebras and suppose that $(\nu_n)_{n\in\N}$ is a normal sequence of measures on $\B^*$.
    
    Then there exists a balanced countable Boolean algebra $\B'\subseteq \bor(C)$ such that $\B\subseteq \B'$ and $(\B^*, \B', (\nu_n)_{n\in\N})$ satisfies $(\mathcal{G}^*)$. 
\end{lemma}

\begin{proof}
Since $\B^*$ is countable, the dual ball in $C(\st(\B^*))$ is metrizable and by the Banach-Alaoglu theorem it is compact in the weak* topology. Hence there is a subsequence $(|\widetilde{\nu}_{n_k}|)_{n\in\N}$ of $(|\widetilde{\nu}_n|)_{n\in\N}$ that converges to a measure $\widetilde{\nu}_\infty$ in the weak* topology. In particular, $(|{\nu}_{n_k}|)_{n\in\N}$ is pointwise convergent to $\nu_\infty$. 

Let us represent $\B$ as an increasing union of finite subalgebras
$$\B=\bigcup_{n\in\N} \B_n.$$
Using Lemma \ref{construction-finite-step} we construct by induction on $k\in\N$ sequences $(m_k)_{k\in\N}$, $(a_k)_{k\in\N}$,\\$(b_k)_{k\in\N}\subseteq \N$ and $ (G_k)_{k\in\N},(H_0^k)_{k\in\N}$, $(H_1^k)_{k\in\N}\subseteq \B^*$  such that 
    \begin{enumerate}
        \item $(m_k)_{k\in\N},(a_k)_{k\in\N},(b_k)_{k\in\N}$ are strictly increasing,
        \item $\forall k\in \N \ \forall n\leqslant k \ \FF\left(\B_n, \bigcup_{i\leqslant k} G_i\right)$ is $(m_n, 2^{-n})$-balanced, 
        \item $\forall k\in \N \ \OO\left(\bigcup_{n \leqslant k} (G_n\cup H_0^n\cup H_1^n)\right)<0.1$,
        \item $\{H_0^k,H_1^k\}_{k\in\N}$ are pairwise disjoint,
        \item $\{G_k\}_{k\in\N}$ are pairwise disjoint,
        \item $G_k\cap H_i^n \neq \varnothing$ if and only if $i=0$ and $n=k$,
        \item $\forall k\in \N \ |\nu_{a_k}|(H_0^k), |\nu_{b_k}|(H_1^k)\geqslant 0.9$,
        \item $\forall k\in\N \ |\nu_{a_k}(G_k\cap H_0^k)|\geqslant 0.3$.       
    \end{enumerate}
Let $k\in \N\cup\{0\}$ and suppose we have constructed  $(m_n)_{n\leqslant k}, (a_n)_{n\leqslant k}, (b_n)_{n\leqslant k}$, $(G_n)_{n\leqslant k}$, $(H_0^n)_{n\leqslant k}$, $(H_1^n)_{n\leqslant k}$ (if $k=0$, then we assume that all of these sequences are empty). We apply Lemma \ref{construction-finite-step} to $\MM=\{\OO\}, \widehat{G}= \bigcup_{i\leqslant k} G_i, \widehat{H}=\bigcup_{i\leqslant k} (H_0^i\cup H_1^i)$  and $ d=\max \{a_k, b_k\}$ (or $\widehat{G}=\widehat{H}=\varnothing$ and $d=1$, if $k=0$) to obtain $m_{k+1}$, $a_{k+1}=a$, $b_{k+1}=b$, $G_{k+1}=G'$, $H^{k+1}_0=H_0$, $H^{k+1}_1=H_1$ satisfying $(1)$-$(8)$.
    
Let $$G=\bigcup_{n\in\N} G_n$$ 
and let $\B'$ be the Boolean algebra generated by $\B\cup \{G\}$. By (2) and Lemma \ref{NN-lemma} $\B'$ is balanced. To see that $(\B^*, \B', (\nu_n)_{n\in\N})$ satisfies $(\mathcal{G}^*)$ we notice that by (6) for every $n\in \N$ we have $G\cap H_i^n= G_n\cap H_i^n$ for $i=0,1$, and so 
\begin{itemize}
    \item $|\nu_{a_n}(G\cap H_0^n)|= |\nu_{a_n}(G_n\cap H_0^n)|\geqslant 0.3$,
    \item $G\cap H_1^n=\varnothing$,
\end{itemize}
which together with (7) and (8) shows that the conditions (a)-(d) of Definition \ref{G*-definition} are satisfied and it completes the proof. 
\end{proof}

Properties of measures on a Boolean algebra $\A$ such as norm or disjointness of Borel supports depend only on countably many elements of $\A$. In particular, the following lemma holds. 

\begin{lemma}\label{disjoint-supports}
  Let $\B= \bigcup_{\alpha<\omega_1} \B_\alpha$, where $(\B_\alpha)_{\alpha<\omega_1}$ is an increasing sequence of countable Boolean algebras. Let $(\nu_n)_{n \in \N}$ be a normal sequence of measures on $\B$. Then there exists $\alpha<\omega_1$ such that for every $\beta> \alpha$ the sequence $(\nu_n\upharpoonright \B_\beta)_{n\in\N}$ is normal.   
\end{lemma}

The next theorem shows how to construct a Boolean algebra with the Grothendieck property and without the Nikodym property under CH.

\begin{theorem}\label{main_CH}\emph{(Talagrand, \cite{talagrand})}
Assume {\sf CH}. There exists a Boolean algebra with the Grothendieck property, but without the Nikodym property. 
\end{theorem}
\begin{proof}
    By Proposition \ref{NN_implies_not_Nikodym} and Proposition \ref{G_implies_Grothendieck} it is enough to construct a balanced Boolean algebra $\B\subseteq \bor(C)$ satisfying $(\mathcal{G})$. We will define $\B$ as a union of a sequence of countable subalgebras $(\B_\alpha)_{\alpha<\omega_1}$ of $\bor(C)$, which is constructed by induction. 

First, using {\sf CH} we fix an enumeration $$(\nu^\alpha, \B^*_\alpha)_{\alpha<\omega_1}$$ of all pairs such that each $\B^*_\alpha$ is a countable subalgebra of $\bor(C)$ and $\nu^\alpha=(\nu^\alpha_n)_{n\in\N}$ is a normal sequence of measures on $\B_\alpha^*$.
We also require each such pair to appear cofinaly often in the sequence $(\nu^\alpha, \B_\alpha^*)_{\alpha<\omega_1}$. 

Successor stage:
Suppose we have constructed $\B_\alpha$. If $\B^*_\alpha$ is not a subalgebra of $\B_\alpha$, then we put $\B_{\alpha+1}=\B_\alpha$. If $\B^*_\alpha\subseteq \B_\alpha$, then by Lemma \ref{one_step_extension} there is a balanced Boolean algebra $\B'\supseteq \B_\alpha$ such that $(\B^*_\alpha, \B', \nu^\alpha)$ satisfies $(\mathcal{G}^*)$. We put $\B_{\alpha+1}=\B'$. 

If $\gamma$ is a limit ordinal then we put $\B_\gamma = \bigcup_{\alpha<\gamma} \B_\alpha$. 

We will prove that $\B=\bigcup_{\alpha<\omega_1} \B_\alpha$ satisfies $(\mathcal{G})$. By Proposition \ref{thm-for-forcing-version} it is enough to show that for every normal sequence $(\nu_n)_{n\in\N}$ of measures on $\B$
there is $\beta<\omega_1$ such that the sequence $({\nu_n\upharpoonright{\B_\beta}})_{n\in\N}$ is normal and $(\B_\beta, \B, (\nu_n\upharpoonright{\B_\beta})_{n\in\N})$ satisfies $(\mathcal{G}^*)$. By Lemma \ref{disjoint-supports} there is $\alpha< \omega_1$ such that for every $\beta\geqslant\alpha$ the sequence $({\nu_n\upharpoonright{\B_\beta}})_{n\in\N}$ is normal. To finish the proof, pick $\beta>\alpha$ such that $\B^*_\beta=\B_\alpha\subseteq\B_\beta$ and $\nu^\beta=\nu\upharpoonright{\B_\alpha}$ and notice that by the construction  $(\B^*_\beta, \B_{\beta+1}, (\nu_n\upharpoonright{\B^*_\beta})_{n\in\N})$ satisfies $(\mathcal{G}^*)$, which implies that $(\B_\beta, \B, (\nu_n\upharpoonright{\B_\beta})_{n\in\N})$ satisfies $(\mathcal{G}^*)$.
\end{proof}

\section{Forcing}\label{forcing}
For the purpose of this section we identify Borel subsets of $C$ with their codes (with respect to some absolute coding, cf. \cite[Section 25]{jech}) i.e. whenever we say about the same Borel sets in different models of {\sf ZFC} we mean Borel sets coded by the same code. 

If $\mathfrak c > \omega_1$, then the method from the previous section does not work, since it requires extending a given Boolean algebra $\mathfrak c$ many times, while this method does not allow us to enlarge uncountable Boolean algebras keeping them balanced. 
Instead, we define a notion of forcing that adds to a given balanced algebra $\B$ a witness for $(\mathcal{G^*})$ for many sequences of measures (chosen by a generic filter) on $\B$ simultaneously. However, it is not possible to pick one extension that is suitable for every sequence, so we will iterate $\omega_1$ such forcings and the final Boolean algebra will have cardinality $\omega_1<\mathfrak c$. 

\begin{definition}\label{forcing-def}
Let $\B\subseteq \bor(C)$ be a balanced countable Boolean algebra containing $\clop(C)$ and fix a representation $\B=\bigcup_{n\in \N} \B_n$, where $(\B_n)_{n\in\N}$ is an increasing sequence of finite subalgebras of $\B$. We define a forcing notion $\PP_\B$ consisting of conditions of the form $$p=(k^p, (m^p_n)_{n\leqslant k^p}, (G^p_n)_{n\leqslant k^p}, (H^p_n)_{n\leqslant k^p}, \MM^p),$$ where
    \begin{enumerate}
        \item $k^p\in \N $,
        \item $(m^p_n)_{n\leqslant k^p}$ is a strictly increasing sequence of natural numbers,
        \item $\MM^p$ is a finite set of probability measures on $\B$ such that $\lambda\upharpoonright\B \in \MM^p$,
        \item $(G^p_n)_{n\leqslant k^p}$ and $(H^p_n)_{n\leqslant k^p}$ are sequences of elements of $\B$ such that 
        \begin{enumerate}
            \item $G^p_n\cap G^p_l = H^p_n\cap H^p_l = G^p_n\cap H^p_l= \varnothing $ for $n\neq l$,
            \item $\mu\left(\bigcup_{n\leqslant k^p} (G^p_n \cup H^p_n) \right) < 0.1$ for all $\mu\in \MM^p$,
            \item $\FF\left(\B_n, \bigcup_{i\leqslant k^p} G^p_i\right) $ is $(m_n, 2^{-n})-$balanced for $n\leqslant k^p$. 
        \end{enumerate}
    \end{enumerate}
    We put $q\leqslant p$, if 
    \begin{itemize}
        \item $k^q\geqslant k^p$,
        \item $m^q_n=m^p_n$ for $n\leqslant k^p$,
        \item $G^q_n=G^p_n$ for $n\leqslant k^p$,
        \item $H^q_n=H^p_n$ for $n\leqslant k^p$,
        \item $\MM^q \supseteq \MM^p$.
    \end{itemize}
    
\end{definition}

As expected, our forcing is ccc. 

\begin{lemma}\label{ccc}
    $\PP_\B$ is $\sigma$-centered. In particular, $\PP_\B$ satisfies ccc.
\end{lemma}
\begin{proof}
 For $p\in \PP_\B$ define $$f(p)=(k^p, (m^p_n)_{n\leqslant k^p}, (G^p_n)_{n\leqslant k^p}, (H^p_n)_{n\leqslant k^p}).$$ If $f(p)=f(q)$, then $r\leqslant p,q$ and $f(r)=f(p)=f(q),$ where $$r=(k^p, (m^p_n)_{n\leqslant k^p}, (G^p_n)_{n\leqslant k^p}, (H^p_n)_{n\leqslant k^p}, \MM^p\cup \MM^q)\in \PP_\B.$$
 In particular, for every $x\in f[\PP_\B]$ the set $f^{-1}(x)$ is directed. Since $\B$ is countable, $f[\PP_\B]$ is also countable, so $\PP_\B$ is a union of countably many directed sets. Hence $\PP_\B$ is $\sigma$-centered.
\end{proof}

The next few lemmas concern the basic properties of $\PP_\B$.

\begin{lemma}\label{simple-extension}
    Let $p\in \PP_\B$ and $k>k^p$. Then there is $q\in \PP_\B, q\leqslant p$ such that $k^q=k$.
\end{lemma}
\begin{proof}
    It is enough to show that there is such $q$ for $k=k^p+1$ and apply an inductive argument. For this put $\widehat{G}= \bigcup_{n\leqslant k^p} G^p_n, \widehat{H}= \bigcup_{n\leqslant k^p} H^p_n$.
    Since $\lambda\upharpoonright \B\in \mathcal{M}^p$, by the condition (4b) of Definition \ref{forcing-def} we have $\lambda(C\backslash (\widehat{G}\cup \widehat{H}))>0$. Hence there is a normal sequence $(\nu_n)_{n\in\N}$ of measures on $\B$, whose supports are included in $C\backslash (\widehat{G}\cup \widehat{H})$. We may also assume that $(|\nu_n|)_{n\in\N}$ is pointwise convergent to a probability measure $\nu_\infty$. In particular, we have 
    $$\nu_\infty(\widehat{G}\cup \widehat{H}) = \lim_{n\to\infty } |\nu_n|(\widehat{G}\cup \widehat{H}) = 0.$$
    Let $\MM=\MM^p\cup\{\nu_\infty\}$. 
    By Lemma \ref{construction-finite-step} there are $m>m_{k^p}$ and $G', H_0\in \B$ such that $$q=(k, ((m^p_n)_{n\leqslant k^p}, m), ((G^p_n)_{n\leqslant k^p},G'), ((H^p_n)_{n\leqslant k^p}, H_0), \MM) \in \PP_\B.$$ We have $q\leqslant p$. 
\end{proof}

\begin{definition}\label{important_names}
Let us introduce the following notation for a balanced Boolean algebra $\B\subseteq \bor(C)$ containing $\clop(C)$:
\begin{itemize}
    \item $\dot{\G}$ denotes the canonical name for a generic filter in $\PP_\B$,
    \item $\dot{G}$ is a $\PP_\B$-name such that $$\PP_\B \forces \dot{G} = \bigcup_{p\in \dot{\G}, n\in \N} \check{G}^p_n, $$
    \item $\dot{\B}'$ denotes a $\PP_\B$-name such that
$$\PP_\B \forces \dot{\B}' \ \text{is the subalgebra of} \ \bor(C) \ \text{generated by} \ \check{\B}\cup \{\dot{G}\}, $$
    \item $\dot{H}, \dot{H}_n$ for $n\in\N$ denote $\PP_\B$-names such that
$$\PP_\B \forces \dot{H}=(\dot{H}_n)_{n\in\N} \ \text{and} \ \forall n\in\N \ \exists p\in \dot{\G} \ \dot{H}_n=\check{H}^p_n.$$
\end{itemize}

\end{definition}
\begin{lemma}\label{G*-first-condition}
    Let $n\in \N$ and $p\in \PP_\B$ be such that $k^p\geqslant n$. Then 
    $$\PP_\B \forces \dot{G} \cap \dot{H}_n \in \check{\B} $$
    and
    $$p\forces \dot{G}\cap \dot{H}_n= \check{G}^p_n\cap \check{H}^p_n.$$

\end{lemma}
\begin{proof}
    By the definitions of $\dot{G}, \dot{H}_n$ and by (4a) from Definition \ref{forcing-def} we have
    \begin{eqnarray*}
        \PP_\B \forces \dot{G}\cap \dot{H}_n = \bigcup_{\substack{q\in \dot{\G}, l\in \N \\k^q\geqslant n,l}} \check{G}^q_l \cap \check{H}^q_n = \bigcup_{\substack{q\in \dot{\G}\\ k^q\geqslant n}}\check{G}^q_n \cap \check{H}^q_n. 
    \end{eqnarray*}
    If $q$ and $r$ are compatible and $k^q, k^r\geqslant n$, then $G^q_n=G^r_n$ and $H^q_n=H^r_n$, so the last union above is in fact a union of one-element family consisting of a set from $\B$. In particular, 
    \begin{eqnarray*}
        \PP_\B \forces \dot{G}\cap \dot{H}_n \in \check{\B} 
    \end{eqnarray*}
and
$$\PP_\B \forces \check{p}\in \dot{\G} \implies \dot{G}\cap\dot{H}_n = \check{G}^p_n \cap \check{H}^p_n $$ 
or equivalently $$p\forces \dot{G}\cap \dot{H}_n= \check{G}^p_n\cap \check{H}^p_n.$$    
\end{proof}

In the following lemma we will make use of Lemma \ref{construction-finite-step} in the case when $\MM$ consists of many probability measures. This is where the differences between our approach and that of work \cite{talagrand} are crucial.

\begin{lemma}\label{one-step-forcing}
Let $\B\subseteq \bor(C)$ be a balanced Boolean algebra containing $\clop(C)$.
    Let $p\in \PP_\B$ and $\nu=(\nu_n)_{n\in \N}$ be a normal sequence of measures on $\B$ such that
    the sequence $(|\nu_n|)_{n\in\N}$ is pointwise convergent on $\B$ to a measure $\OO\in \MM^p$.
 Then 
     \begin{enumerate}
         \item $p \forces  \forall k\in \N \ \exists n,l>k \ |\check{\nu}_n|(\dot{H}_l)\geqslant 0.9 \ \text{and} \ |\check{\nu}_n(\dot{G}\cap \dot{H}_l)|\geqslant 0.3$,
         \item $p \forces  \forall k\in \N \ \exists n,l>k \ |\check{\nu}_n|(\dot{H}_l)\geqslant 0.9 \ \text{and} \ \dot{G}\cap \dot{H}_l = \varnothing $.
     \end{enumerate}
    
\end{lemma}
 \begin{proof}
In the light of Lemma \ref{G*-first-condition} it is enough to show that the sets $$\mathbb{D}_k=\{q\in \PP_\B : k^q>k, \exists n>k \ |\nu_n|(H^q_{k^q})\geqslant0.9 \ \text{and} \ |\nu_n(G^q_{k^q}\cap H^q_{k^q})|\geqslant 0.3\}$$
and
$$\mathbb{E}_k=\{q\in \PP_\B : k^q>k, \exists n>k \ |\nu_n|(H^q_{k^q})
\geqslant0.9 \ \text{and} \ G^q_{k^q}\cap H^q_{k^q} = \varnothing\}$$
are dense below $p$ for every $k\in \N$.

First, we will show that $\mathbb{D}_k$ is dense below $p$. 
Pick any $r\leqslant p$. By Lemma \ref{simple-extension} we may assume that $k^r\geqslant k$. 

We apply Lemma \ref{construction-finite-step} to $\MM=\MM^r, \widehat{G}= \bigcup_{i\leqslant k^r} G_i^r, \widehat{H}=\bigcup_{i\leqslant k^r} H_i^r$  and $ d=k$ to obtain $m=m_{k^r+1}$, $a>k$, $G'$, $H_0$ so that the conditions $(1)$-$(8)$ from Lemma \ref{construction-finite-step} are satisfied. In particular:
\begin{enumerate}[label=(\alph*)]
    \item $m>m_{k^r}$,
    \item $G'\cap G_i^r= H_0\cap H_i^r=G'\cap H_i^r=\varnothing$ for $i\leqslant k^r$,
    \item $\mu\left(\bigcup_{i\leqslant k^r} (G_i^r \cup H_i^r) \cup G'\cup H_0\right) <0.1$ for all $\mu \in \MM^r$,
    \item $\FF\left(\B_n, \bigcup_{i\leqslant k^r} G^r_i \cup G'\right) $ is $(m_n, 2^{-n})-$balanced for $n\leqslant k^r+1$, 
\item $|\nu_a|(H_0)\geqslant0.9$,
\item $|\nu_a(G'\cap H_0)|\geqslant 0.3$.
\end{enumerate}
It follows from (a)-(d) that
 $$q=(k^r+1, ((m^r_i)_{i\leqslant k^r}, m), ((G^r_i)_{i\leqslant k^r},G'), ((H^r_i)_{i\leqslant k^r}, H_0), \MM^r) \in \PP_\B$$ and from (e), (f) that $q\in \mathbb{D}_k$.

We show the density of $\mathbb{E}_k$ in a similar way: the difference is that instead of $H_0$ we pick $H_1$ such that 
\begin{itemize}
    \item $|\nu_a|(H_1)\geqslant0.9$,
    \item $G'\cap H_1 = \varnothing$.
\end{itemize}
 \end{proof}
Directly from Lemma \ref{G*-first-condition}, Lemma \ref{one-step-forcing} and Definition \ref{G*-definition} we obtain
 
 \begin{proposition}\label{one-step-thm}
 Let $\B\subseteq \bor(C)$ be a balanced Boolean algebra that contains $\clop(C)$.
          Let $p\in \PP_\B$ and $\nu=(\nu_n)_{n\in \N}$ be a normal sequence of measures on $\B$ such that the sequence $(|\nu_n|)_{n\in\N}$ is pointwise convergent to a measure $\nu_\infty$.
     Suppose that $\OO\in \MM^p$. Then  $$ p\forces (\check{\B}, \dot{\B}', \check{\nu}) \ \text{satisfies} \ (\mathcal{G}^*).$$
 \end{proposition}

\begin{proposition}\label{forcing-NN}
    Suppose that $\B\subseteq \bor(C)$ is a balanced countable Boolean algebra containing $\clop(C)$. Then $$\PP_\B \forces \dot{\B}' \ \text{is balanced}.$$ 
\end{proposition}
\begin{proof}
    Since the property of being balanced is absolute between transitive models of {\sf ZFC} 
    $$V^{\PP_\B} \models \B \ \text{is balanced}. $$
In $V^{\PP_\B}$ for every $n\in \N$ we define $G_n=G^p_n$ for some $p\in \G$ such that $k^p\geqslant n$. Then $\B'$ is the Boolean algebra generated by $\B\cup \{G\}$, where $G=\bigcup_{n\in \N} G_n$. By Definition \ref{forcing-def}(4a,4c) the hypothesis of Lemma \ref{NN-lemma} is satisfied and so $\B'$ is balanced. 
\end{proof}

\begin{definition}\label{iteration-definition}
    We define an iteration $(\PP_\alpha)_{\alpha\leqslant\omega_1}$ with finite supports and $\PP_\alpha$-names $\dot{\B}_\alpha$ for every $\alpha\leqslant\omega_1$ by induction in the following way:
    \begin{itemize}
        \item $\PP_0$ is the trivial forcing and $\B_0=\clop(C)$,
        \item having constructed $\PP_\alpha$ and $\dot{\B}_\alpha$ we define 
        $$\PP_{\alpha+1}= \PP_{\alpha} * \PP_{\dot{\B}_\alpha}$$ 
        and we pick a $\PP_{\alpha+1}$-name $\dot{\B}_{\alpha+1}$ such that 
        $$\PP_{\alpha+1} \forces \dot{\B}_{\alpha+1}= \dot{\B}'_\alpha,$$
        \item if $\gamma$ is a limit ordinal, then we define $\PP_\gamma$ as the iteration of $(\PP_\alpha)_{\alpha<\gamma}$ with finite supports and we pick $\dot{\B}_\gamma$ so that 
        $$\PP_\gamma \forces \dot{\B}_\gamma = \bigcup_{\alpha<\gamma} \dot{\B}_\alpha $$
    \end{itemize}
\end{definition}

We will identify each $\PP_\alpha$ with the subset of $\PP_{\omega_1}$ consisting of those $p\in \PP_{\omega_1}$, for which $p(\beta)=\mathbbm{1}_\beta$ for all $\beta\geqslant\alpha$, where $\mathbbm{1}_\beta$ 
denotes the maximal element of $\PP_{\dot{\B}_\beta}$.
\begin{lemma}
    $\PP_{\omega_1}$ is $\sigma$-centered. In particular, $\PP_{\omega_1}$ satisfies ccc. 
\end{lemma}
\begin{proof}
    This follows from Lemma \ref{ccc} and the fact that a finite support iteration of length $\omega_1$ of $\sigma$-centered forcings is $\sigma$-centered \cite[proof of Lemma 2]{Tall}.
\end{proof}

Since $\PP_{\omega_1}$ satisfies ccc, the standard closure argument shows the following (cf. \cite[Lemma 5.3]{fajardo})
\begin{lemma}\label{decidicion-club}
     Let $\dot{\nu}=(\dot{\nu}_n)_{n\in \N}$ be a sequence such that 
    $$\PP_{\omega_1} \forces (\dot{\nu}_n)_{n\in \N} \  \text{is a sequence of measures on} \ \dot{\B}_{\omega_1}.$$
Let $$C_{\dot{\nu}} = \{\alpha<\omega_1: \PP_{\omega_1} \forces \dot{\nu}\upharpoonright \dot{\B}_\alpha \in V^{\PP_\alpha} \}.$$ 
    Then $C_{\dot{\nu}}$ is a closed and unbounded subset of $\omega_1$. 
\end{lemma}

\begin{proposition}
    $\PP_{\omega_1} \forces \dot{\B}_{\omega_1}$ is balanced and satisfies $(\mathcal{G})$.
\end{proposition}

\begin{proof}
    The fact that $\PP_{\omega_1} \forces \dot{\B}_{\omega_1}$  is balanced follows directly from Proposition \ref{forcing-NN} and the fact that increasing unions of balanced Boolean algebras are balanced. 

    To prove the second part of the proposition, by Proposition \ref{thm-for-forcing-version} it is enough to show that for every sequence $(\dot{\nu}_n)_{n\in \N}$ such that 
     $$\PP_{\omega_1} \forces (\dot{\nu}_n)_{n\in \N} \ \text{is a normal sequence of measures on} \ \dot{\B}_{\omega_1}$$ 
 we have
$$\PP_{\omega_1} \forces \exists_{\alpha<\omega_1} \ (\dot{\B}_\alpha, \dot{\B}_{\omega_1}, \dot{\nu}\upharpoonright{\dot{\B}_\alpha}) \ \text{satisfies} \ (\mathcal{G}^*).$$
   Pick any $p\in\PP_{\omega_1}$. By Lemma \ref{disjoint-supports} there is $\alpha_1<\omega_1$ and $p_1\leqslant p$ such that for every $\beta\geqslant \alpha_1$ 
   $$p_1\forces (\dot{\nu}_n\upharpoonright{\dot{\B}_\beta})_{n\in\N} \ \text{is normal}.$$
By Lemma \ref{decidicion-club} there are: $\beta\in C_{\dot{\nu}}$ such that $\alpha_1<\beta$, $p_1\in V^{\PP_\beta}$ and $\PP_{\beta}$-names $\dot{\nu}^\beta, \dot{\nu}^\beta_n$ for $n\in \N$ such that
$$ \PP_\beta \forces \dot{\nu}^\beta = (\dot{\nu}^\beta_n)_{n\in\N}=(\dot{\nu}_n\upharpoonright{\dot{\B}_\beta})_{n\in\N}. $$
Without loss of generality by passing to a subsequence we may assume that there is $\dot{\nu}_\infty$ such that
$$ \PP_\beta \forces (|\dot{{\nu}}^\beta_n|)_{n\in\N} \ \text{is pointwise convergent to a measure} \ \dot{{\nu}}_\infty. $$
   Since $p_1(\beta)=\mathbbm{1}_{\beta}$, there is $p_2\leqslant p_1$ such that $$p_2\upharpoonright \beta \forces \dot{\nu}_\infty \in \dot{\MM}^{\check{p}_2({\beta})}.$$ 
   By Proposition \ref{one-step-thm} we have 
   $$p_2 \forces (\dot{\B}_{\beta}, \dot{\B}_{\beta+1}, \dot{\nu}^\beta) \ \text{satisfies} \ (\mathcal{G}^*)$$
   and hence $$p_2 \forces \exists_{\alpha<\omega_1} \ (\dot{\B}_\alpha, \dot{\B}_{\omega_1}, \dot{\nu}\upharpoonright{\dot{\B}_\alpha}) \ \text{satisfies} \ (\mathcal{G}^*)$$
   which completes the proof. 
\end{proof}
In particular, by Proposition \ref{NN_implies_not_Nikodym} and Proposition \ref{G_implies_Grothendieck} we obtain
\begin{corollary}\label{forcing_corollary}
    $\PP_{\omega_1} \forces$ there is a Boolean algebra of size $\omega_1$ with the Grothendieck property but without the Nikodym property. 
\end{corollary}

The existence of a Boolean algebra with the Grothendieck property of small cardinality has influence on certain cardinal characteristics of the continuum. Below $\mathfrak p$ denotes the pseudointersection number, $\mathfrak s$ is the splitting number and $\cov(\MM)$ is the covering number of the ideal of meager sets in $\R$. 

\begin{corollary}\label{pseudointersection_number}
    $\PP_{\omega_1} \forces \mathfrak p= \mathfrak s = \cov(\MM)=\omega_1$.
\end{corollary}

\begin{proof}
    Apply \cite[Corollary 4.3]{Sobota_kukuryku}.
\end{proof}

\begin{theorem}\label{main}
    It is consistent with $\neg${\sf CH} that there is a Boolean algebra of size $\omega_1$ with the Grothendieck property but without the Nikodym property. 
\end{theorem}
\begin{proof}
    Start with a model $V$ of {\sf ZFC} satisfying $\neg${\sf CH}. Since $\PP_{\omega_1}$ is $\sigma$-centered, it preserves cardinals and the value of the continuum, so we have 
    $$V^{\PP_{\omega_1}}\models \mathfrak c = \mathfrak{c}^V> \omega_1^V=\omega_1.$$
By Corollary \ref{forcing_corollary} in $V^{\PP_{\omega_1}}$ there is a Boolean algebra with the Grothendieck property, but without the Nikodym property.
\end{proof}

\section{Final remarks}
Let us start with a comment concerning differences between the original Talagrand's contruction and our approach. 
To obtain the Grothendieck property, Talagrand uses {\sf CH} to enumerate (in a sequence of length $\omega_1$) all normalized sequences $(\nu_n)_{n\in\N}$ of measures on countable subalgebras of Borel subsets of the Cantor set, for which there exists an antichain $(H_n)_{n\in\N}$ such that $|\nu_n|(H_n)\geqslant 0.95$. Then for each such sequence he constructs another antichain $(G_n)_{n\in\N}$ satisfying the hypothesis of Lemma \ref{NN-lemma}, such that for $G=\bigcup_{n\in\N} G_n$ we have 
\begin{itemize}
    \item $|\nu_n(G\cap H_n)|\geqslant 0.4$ for infinitely many $n\in \N$, 
    \item $|\nu_n(G\cap H_n)|< 0.1$ for infinitely many $n\in \N$. 
\end{itemize}
It follows that extending a given Boolean algebra with $G$ keeps it balanced, and in the extension the sequence $(\nu_n)_{n\in\N}$ satisfies a property similar to $(\mathcal{G^*})$ from Definition \ref{G*-definition}. Thus, the final algebra has the Grothendieck property and does not have the Nikodym property. The same approach was used in \cite{SZpreprint}.

However, this technique applies only when we work with one sequence of measures at a time. In our method, we construct a suitable antichain $(H_n)_{n\in\N}$ along with $(G_n)_{n\in\N}$, which allows us to pick both the antichains in a generic way, making them working for uncountably many sequences of measures simultaneously.

The method of construction we have described relies strongly on the fact that the Boolean algebras we extend are countable.

\begin{question}
    Let $\B\subseteq\bor(C)$ be a balanced Boolean algebra of cardinality $<\mathfrak c$ and $\nu$ be a normal sequence of measures on $\B$.
    
    Does there exist a balanced Boolean algebra $\B\subseteq \B'\subseteq \bor(C)$ such that $(\B,\B',\nu)$ satisfies $(\mathcal{G}^*)$?
\end{question}
Positive answer for the above question would allow us to construct (by induction of length $\mathfrak c$) a balanced Boolean algebra of size $\mathfrak c$. Thus, it would imply the positive answer for the following question:
\begin{question}
    Does there exist (in {\sf ZFC}) a balanced Boolean algebra with the Grothendieck property?
\end{question}
One may look for candidates for Boolean algebras with the Grothendieck and without Nikodym property among maximal balanced Boolean algebras. 
\begin{question}
    Let $\B\subseteq \bor(C)$ be a maximal balanced Boolean algebra. Does $\B$ have the Grothendieck property? 
\end{question}

\section*{Acknowledgments}
We would like to thank our PhD supervisors Piotr Koszmider and Piotr Borodulin-Nadzieja for introducing to the topic, careful reading and valuable suggestions. 
We would like to express our gratitude towards The Fields Institute for
Research in Mathematical Sciences in Toronto for the hospitality and excellent conditions for research during Thematic Program on Set Theoretic Methods in Algebra, Dynamics and Geometry (January - June, 2023), during which most of the results contained in this work were obtained. We would also like to thank Omar Selim for sharing his notes, Julia Millhouse and Forte Shinko for the language assistance with French.

\bibliographystyle{amsplain}
\bibliography{bibliography}
\end{document}